\documentclass[11pt]{amsart}
\usepackage[T1]{fontenc}
\usepackage{times,mathptm}
\usepackage{amssymb,verbatim}
\usepackage[all]{xy}
\usepackage{subfig}
\usepackage{tikz}
\usepackage{color}
\usepackage{a4wide}
\usepackage[normalem]{ulem}
\usepackage{hyperref}
\usepackage{wrapfig}
\usetikzlibrary{arrows,patterns}

\newcommand{\R}{\mathbb{R}}
\newcommand{\C}{\mathbb{C}}

\newcommand\Z{\mathbb{Z}}
\newcommand\Prym{\textrm{Prym}}
\newcommand{\N}{\mathbb{N}}

\newcommand{\SL}{{\rm SL}}
\newcommand{\GL}{{\rm GL}}

\newcommand{\ol}{\overline}

\newcommand{\Ord}{\mathcal{O}}

\newcommand{\Pcal}{\mathcal{P}}

\newcommand{\Id}{\mathrm{Id}}
\newcommand{\SD}{\mathcal{S}^h_D}

\renewcommand\mod{\text{ mod }}
\newcommand{\DS}{\displaystyle}
\newcommand{\inv}{\tau}

\newcommand{\Aa}{\mathrm{Area}}
\newcommand{\Cc}{\mathcal{C}}
\newcommand{\Sc}{\mathcal{S}}

\newcommand{\Bc}{\mathcal{B}}
\newcommand{\Hc}{\mathcal{H}}
\newcommand{\PP}{\mathbf{P}}
\newcommand{\Jac}{\mathbf{Jac}}
\newtheorem{Theorem}{Theorem}[section]
\newtheorem{Corollary}[Theorem]{Corollary}
\newtheorem{Lemma}[Theorem]{Lemma}
\newtheorem{Proposition}[Theorem]{Proposition}
\newtheorem{Remark}[Theorem]{Remark}

\begin{document}
\title[Prym eigenforms in genus four]  {Weierstrass Prym eigenforms in genus four}

\author{Erwan Lanneau, Duc-Manh Nguyen}


\address{
UMR CNRS 5582 \newline
Univ. Grenoble Alpes, CNRS, Institut Fourier, F-38000 Grenoble, France}
\email{erwan.lanneau@univ-grenoble-alpes.fr}

\address{
UMR CNRS 5251 \newline
IMB Bordeaux-Universit\'e de Bordeaux \newline
351, Cours de la Lib\'eration \newline
33405 Talence Cedex, France}

\email{duc-manh.nguyen@math.u-bordeaux.fr}



\maketitle

\section{Introduction}
Let $\Hc(6)$ denote the space of pairs $(X,\omega)$, where $X$ is a Riemann surface of genus four and $\omega$
is a holomorphic $1$-form on X having a single zero. Following~\cite{Mc7}, $\Prym(6)$ is the subset of $\Hc(6)$ where $X$ admits a holomorphic involution (Prym involution) $\inv$ which has exactly two fixed points and satisfies $\inv^*\omega=-\omega$. We will call such pairs {\em Prym forms}. 
The space of holomorphic $1$-forms $\Omega(X)$ on $X$ splits into $\Omega^-(X,\inv)\oplus\Omega^+(X,\inv)$ where
$\Omega^-(X,\inv)$ is the eigenspace of the eigenvalue $-1$. Similarly one has 
$H^-(X;\Z)=\{c \in H_1(X,\Z), \; \inv_*c=-c\}$. Define $\PP(X,\inv)=(\Omega^-(X,\inv))^*/H_1^-(X;\Z)$.
By definition $\PP(X,\inv)$ is a sub-abelian variety of $\Jac(X)$. We will call it the {\em Prym variety} of $X$.
By assumption we have $\dim_\C \PP(X,\inv)=2$.

Recall that a {\em discriminant} is a positive integer congruent to $0$ or $1$ modulo $4$.
The quadratic order with discriminant $D$ is denoted by $\Ord_D$.
We have $\Ord_D \simeq \Z[x]/(x^2+bx+c)$, for any $(b,c) \in \Z^2$ such  that $b^2-4c=D$.
For each discriminant $D$, we define $\Omega E_D(6)$ the subset of $(X,\omega) \in \Prym(6)$ such that
\begin{itemize}
 \item[(1)] $\PP(X,\inv)$ admits a real multiplication by the quadratic order $\Ord_D$, and

 \item[(2)] $\omega$ is an eigenvector for the action of $\Ord_D$.
\end{itemize}
Elements of $\Omega E_D(6)$ are called {\em Prym eigenforms} in $\Hc(6)$. For a more detailed definition, we refer to \cite{Mc7,Lanneau:Manh:H4}.
In \cite{Mc7}, McMullen showed that the locus $\Omega E_D(6)$ is a finite union of closed $\GL^+(2,\R)$-orbits.
The geometry of these affine invariant subvarieties has been recently investigated in~\cite{Moller2011,TZ2016,TZ2017,Zac2017}.
The main goal of this paper is to complete this description.

\begin{Theorem}\label{thm:H6:eig:comp}
For any discriminant $D \not\in \{4,9\}$, the locus $\Omega E_D(6)$ is non empty and connected.
\end{Theorem}
We will see that $\Omega E_4(6)$ and $\Omega E_9(6)$ are empty.\\
A square-tiled surface is a form $(X,\omega)$ such that $\omega(\gamma) \in \Z\oplus\imath\Z$
for any $\gamma \in H_{1}(X,\Sigma,\Z)$, where $\Sigma$ is the zero set of $\omega$.
For such a surface, integration of the form $\omega$
gives a holomorphic map $X \rightarrow \C /\Z^{2}$ which can be normalized so that it is branched only
above the origin. The $n$ preimages of the square $[0,1]^{2}$ provide a tiling of the surface $X$.
We say that $(X,\omega)$ is {\em primitive} if
$$
\Lambda(X,\omega):=\left\{ \omega(\gamma),\  \gamma \in H_{1}(X,\Sigma,\Z)\right\}= \Z\oplus i\Z.
$$
$\GL^+(2,\R)$ acts naturally on $\mathcal {T}_n$ the set of degree $n$, primitive square-tiled surfaces in $\Prym(6)$.
Along the line we will also prove the following theorem for the topology of the branched covers:
\begin{Theorem}\label{thm:H6:eig:comp:2}
Let $n\in\Z$ be any integer. If $n\geq 8$ is even then there is exactly one $\GL^+(2,\R)$-orbit in $\mathcal T_n$. 
Otherwise $\mathcal T_n$ is empty.
\end{Theorem}
Theorem~\ref{thm:H6:eig:comp:2} generalizes previous result by~\cite{Mc4,HLelievre,Lanneau:Manh:H4}.
\subsection*{Outline}
This paper is very much a continuation of \cite{Lanneau:Manh:H4} in which we announced a weaker version of Theorem~\ref{thm:H6:eig:comp}.
This weaker result is obtained by using tools and techniques similar to the ones developed in \cite{Lanneau:Manh:H4} (see also \cite{Mc4}). However, because of some new phenomena in genus four, those tools are not sufficient to obtain Theorem~\ref{thm:H6:eig:comp}. We will give below an overview of our strategy to prove Theorems~\ref{thm:H6:eig:comp} and~\ref{thm:H6:eig:comp:2}.

\begin{enumerate}
 \item We start by showing that every $\GL^+(2,\R)$-orbit in $\Omega E_D(6)$ contains a horizontally periodic surface with 4 horizontal cylinders (cf. Lemma~\ref{lm:4cyl:dec:exist}).
 We then show that up to some renormalization by $\GL^+(2,\R)$, one can encode the corresponding cylinder decomposition by parameters called {\em prototypes}
 (cf. Proposition~\ref{prop:normalize:A}).
 For a fixed discriminant $D$, the set of prototypes is denoted by $\Pcal_D$.
 Note that $\Pcal_D$ is a finite set.

 \item There are two different diagrams, called Model A and Model B, for $4$-cylinder decompositions of surfaces in $\Prym(6)$. Therefore, the set of prototypes $\Pcal_D$ is naturally split into two disjoint subsets $\Pcal^A_D$ and $\Pcal^B_D$ according to the associated diagram.

\item We next introduce the Butterfly move transformations on the set $\Pcal^A_D$ (cf. Proposition~\ref{prop:BM:transform:rule}). Those transformations encode the switches from a $4$-cylinder decomposition in Model A to another $4$-cylinder decomposition in Model A on the same surface.
 We will call an equivalence class of the relation generated by the Butterfly moves in $\Pcal^A_D$ a {\em component} of $\Pcal^A_D$.
 By construction, surfaces associated with prototypes in the same component belong to the same $\GL^+(2,\R)$-orbit.
 Thus we obtain an upper bound for the number of $\GL^+(2,\R)$-orbits in $\Omega E_D(6)$ by the number of components of $\Pcal^A_D$.

 \item Using a similar strategy to the one used in \cite{Lanneau:Manh:H4} and \cite{Mc4}, one can classify the components of $\Pcal^A_D$ for $D$ large enough (cf. Theorem~\ref{theo:H6:connect:PD}).
 This classification reveals that $\Pcal_D^A$ has two components when $D$ is even or $D \equiv 1 \mod 8$.
 While the disconnectedness of $\Pcal^A_D$ for $D$ even can be  easily seen, the disconnectedness of $\Pcal^A_D$ for $D\equiv 1 \mod 8$ is somewhat more subtle (cf. Lemma~\ref{lm:no:conn:PAD:even} and Lemma~\ref{lm:D1mod8:S1:no:connect:S2}). This new phenomenon did not occur in genus two and three.

 Theorem~\ref{theo:H6:connect:PD} implies immediately that $\Omega E_D(6)$ is connected if $D\equiv 5 \mod 8$ (when $D$ is large enough).
 However, to our surprise, for the remaining values of $D$, the number $\GL^+(2,\R)$-orbits in $\Omega E_D(6)$ is not equal to the number of the components of $\Pcal^A_D$.
 This is another striking difference between genus four and genus two and three.

 \item To obtain Theorem~\ref{thm:H6:eig:comp} for $D$ even and $D \equiv 1 \mod 8$, one needs to connect two components of $\Pcal^A_D$.
 For this purpose, we  will introduce  new transformations on the set of prototypes.

 A prototype in $\Pcal_D$ is a quadruple  of integers $(w,h,t,e)$ satisfying some specific conditions depending on $D$ (see Proposition~\ref{prop:normalize:A}).
 Given a horizontally periodic surface in $\Omega E_D(6)$, it is generally difficult to determine all the parameters of the prototype of the cylinder decomposition in another periodic direction.
 Nevertheless, one important parameter, namely $e$, of this prototype can be computed quite easily (cf. Lemma~\ref{lm:cyl}).
 This new tool turns out to be an essential ingredient of our proofs.
 In what follows, we will only consider $D$ large enough  such that the generic statements of Theorem~\ref{theo:H6:connect:PD} hold.

 \begin{itemize}
  \item {\bf Case $D$ even:} The two components of $\Pcal^A_D$ are distinguished by the congruence class of $e$ modulo 4.
  To connect the two components of $\Pcal^A_D$, it suffices to construct a surface which admits  4-cylinder decompositions in Model A in two different directions,
  such that the corresponding $e$-parameters are not congruent modulo 4.
  For the case $D$ is even and not a square number, we  make use of 4-cylinder decomposition in Model B, and new transformations called {\em switch moves}, which correspond to passages from a cylinder decomposition in Model B to a cylinder decomposition in Model A. We will show that one can always find a suitable prototypical surface in Model B, and two switch moves among the four introduced in Proposition~\ref{prop:switch:move}, such that the prototypes of the  new periodic directions  belong to different components of $\Pcal^A_D$.
 For $D$ is an even square number, we will use 2-cylinder decompositions and adapted switch moves to get the same conclusion. Details are given in Sections~\ref{sec:pf:H6:conn:D:even:n:sq}, and~\ref{sec:main:thm:D:square:even}.

 \item {\bf Case $D\equiv 1 \mod 8$:} We denote the two components of $\Pcal^A_D$ by $\Pcal_D^{A_1}$ and $\Pcal_D^{A_2}$ (see Theorem~\ref{theo:H6:connect:PD}).
 The two components $\Pcal^{A_i}_D$ can not be  distinguished only by the $e$-parameter in general.
 However, there is a simple sufficient (but not necessary) condition on the $e$-parameter which allows us to conclude that the prototype belongs to $\Pcal_D^{A_1}$ but not $\Pcal^{A_2}_D$.
 In view of this observation, to prove Theorem~\ref{thm:H6:eig:comp} in this case, we construct a prototypical surface from a suitable prototype in $\Pcal_D^{A_2}$ and show that this surface admits a cylinder decomposition in Model A with associated prototype in $\Pcal_D^{A_1}$. Details are given in Section~\ref{sec:D1mod8}.
 \end{itemize}

 \item For small (and exceptional) values   of $D$, Theorem~\ref{thm:H6:eig:comp} are proved ``by hand'' with computer assistance.

 \item Theorem~\ref{thm:H6:eig:comp:2} is a direct consequence of Theorem~\ref{thm:H6:eig:comp} and the fact that a primitive square-tiled surface in $\Prym(6)$ belongs to $\Omega E_{d^2}(6)$ if and only if it is constructed from $2d$ unit squares (see Prop~\ref{prop:square:tiled}).
\end{enumerate}

\subsection*{Acknowledgements:} The authors warmly thank Jonathan Zachhuber and David Torres for helpful conversations.
This work was partially supported by the ANR Project GeoDyM and the Labex Persyval.

\section{Cylinder decompositions and the space of prototypes}\label{sec:prototypes}

The main goal of this section is to provide a canonical representation
of any four cylinder decomposition of a surface in $\Omega  E_{D}(6)$ in
terms    of    {\it    prototype}. We will also define an equivalence relation $\sim$ on the set of prototypes
such that the number of $\GL^+(2,\R)$-orbits in $\Omega E_D(6)$ is bounded by the number of equivalence classes of $\sim$.

\subsection{Four-cylinder decompositions}
Recall that a cylinder is called {\em simple} if each of its boundary consists of a single saddle connection.
We will call a cylinder {\em semi-simple} if one of its boundary components consists of a single saddle connection.
If it is not simple, then we will call it {\em strictly semi-simple}.
We first show

\begin{Lemma}\label{lm:4cyl:dec:exist}
 Let $(X,\omega)$ be a translation surface in $\Omega E_D(6)$ for some discriminant $D$. Then $(X,\omega)$ admits a 4-cylinder decomposition.
\end{Lemma}
\begin{proof}
 By \cite{Mc7}, we know that $(X,\omega)$ is a Veech surface, hence it admits decompositions into cylinders in infinitely many directions.
 Recall that the Prym involution of $X$ has a unique regular fixed point. Thus, a cylinder cannot be invariant by this involution.
 It follows that there are either 2 or 4 cylinders in each cylinder decomposition.

 Suppose that $(X,\omega)$ admits a 2-cylinder decomposition in the horizontal direction.
 Let us denote the two horizontal cylinders by $C_1,C_2$.
 By inspecting all the possible configurations of the horizontal saddle connections,
 we see that for each $i\in \{1,2\}$, there is a saddle connection which is contained in both boundary components of $C_i$.
 Thus, there is a simple cylinder $C$ which is filled by simple closed geodesics  represented by geodesic segments joining a point in the bottom border of $C_i$ and a point in the top border of $C_i$. Since $(X,\omega)$ is a Veech surface, it admits a cylinder decomposition in the direction of $C$.
 Since $C$ is a simple cylinder, there must be 4 cylinders in this decomposition.
\end{proof}

\subsection{Space of prototypes}

The surfaces in $\Omega E_D(6)$ admit two types of decomposition into four cylinders,
which will be called Model $A$, and Model $B$. The Model $A$ is characterized by the presence of
simple cylinders, while the Model $B$ is characterized by the
presence of strictly semi-simple cylinders (see Figure~\ref{fig:H6:ModAB}).

The next proposition is analogous to \cite[Prop 4.2, 4.5]{Lanneau:Manh:H4}.

\begin{Proposition}
\label{prop:normalize:A}
%
Let $(X,\omega) \in \Omega E_D(6)$  be a Prym eigenform which admits a
cylinder    decomposition  with 4-cylinders, equipped with
the   symplectic   basis    presented   in Figure~\ref{fig:H6:ModAB}.
Then   up    to    the   action $\mathrm{GL}^+(2,\R)$ and Dehn twists
there exists $(w,h,t,e) \in \Z^{4}$ such that
\begin{enumerate}
\item   the          tuple         $(w,h,t,e)$         satisfies
  $(\mathcal{P}_D)\quad \left\{
  \begin{array}{l}
   w>0,h>0,\;        0\leq   t<\gcd(w,h),\\
  \gcd   (w,h,t,e)    =1,\\
 D=e^2+4wh,\\
 0< \lambda:=\frac{e+\sqrt{D}}{2}< w  \textrm{ and } \lambda \neq w/2\\
\end{array}
\right.$

\item There exists  a generator $T$ of $\Ord_D$ whose the matrix,
in  the  basis  $\{\alpha_1,\beta_1,\alpha_2,\beta_2\}$,  is
  $\left(%
\begin{smallmatrix}
  e & 0 & w & t \\ 0 & e & 0 & h \\ h & -t & 0 & 0 \\ 0 & w & 0 & 0
  \\
\end{smallmatrix}%
\right)$.
\item $T^*(\omega)=\lambda\omega$,
\item \label{normalize:A} In these coordinates
$$\left\{\begin{array}{l}
  \omega(\Z\alpha_{2,1}+\Z\beta_{2,1})=\omega(\Z\alpha_{2,2}+\Z\beta_{2,2})=\Z(\frac{w}{2},0)+\Z(\frac{t}{2},\frac{h}{2})
  \\ \omega(\Z\alpha_{1,1}+\Z\beta_{1,1})=\omega(\Z\alpha_{1,2}+\Z\beta_{1,2})=\frac{\lambda}{2}\cdot
  \Z^{2}
\end{array}
\right.
$$
\end{enumerate}

\noindent Conversely,  let $(X,\omega) \in \mathcal{H}(6)$ having a four-cylinder decomposition.
Assume there exists $(w,h,t,e)  \in \Z^4$ satisfying
$(\Pcal_D)$,  such  that  after  normalizing  by  $\GL^+(2,\R)$,  all  the
conditions in $(\ref{normalize:A})$  are fulfilled. Then $(X,\omega) \in \Omega E_D(6)$. 
\end{Proposition}
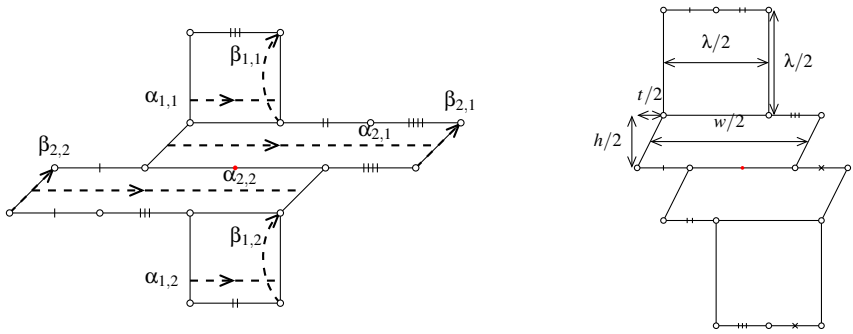
\begin{figure}[htbp]
\begin{minipage}[l]{0.23\linewidth}
\centering \subfloat{%
\begin{tikzpicture}[scale=0.6]

\draw (0,0) -- (-2,0) -- (-3,-1) -- (1,-1) -- (1,-3) -- (3,-3) -- (3,-1) -- (4,0) -- (6,0)
-- (7,1) -- (3,1) -- (3,3) --  (1,3) -- (1,1) -- cycle;

\draw (1,1) -- (3,1) (0,0) -- (4,0) (1,-1) -- (3,-1);

\draw[thick, dashed, ->, >= angle 45] (1,1.5) -- (2,1.5); \draw[thick,
  dashed] (2,1.5)  -- (3,1.5); \draw[thick,  dashed, ->, >=  angle 45]
(3,1) .. controls (2.5,1.5) and (2.5,2.5) .. (3,3);

\draw[thick,  dashed,   ->,  >=   angle  45]  (0.5,0.5)   --  (3,0.5);
\draw[thick, dashed] (3,0.5) -- (6.5,0.5);
\draw[thick, dashed, ->, >=  angle 45] (6,0) -- (7,1);

\draw[thick,  dashed,   ->,  >=   angle  45]  (-2.5,-0.5)   --  (0,-0.5);
\draw[thick, dashed] (0,-0.5) -- (3.5,-0.5);
\draw[thick, dashed, ->, >=  angle 45] (-3,-1) -- (-2,0);

\draw[thick,  dashed,   ->,  >=   angle  45]  (1,-2.5)   --  (2,-2.5);
\draw[thick, dashed] (2,-2.5) -- (3,-2.5);

\draw[thick, dashed, ->, >= angle 45] (3,-3) .. controls (2.5,-2.5) and (2.5,-1.5) .. (3,-1);

\filldraw[fill=white,draw=black] (0,0) circle (2pt)  (4,0)  circle (2pt)  (6,0)
circle (2pt) (7,1) circle (2pt)  (5,1) circle (2pt) (3,1) circle (2pt)
(3,3) circle (2pt) (1,3) circle (2pt) (1,1) circle (2pt) (-2,0) circle (2pt)
(-3,-1) circle (2pt) (-1,-1) circle (2pt) (1,-1) circle (2pt) (1,-3) circle (2pt)
(3,-3) circle (2pt) (3,-1) circle (2pt);

\draw   (2,-3)   +(-0.05,-0.1)   --  +(-0.05,0.1)   +(0.05,-0.1)   -- +(0.05,0.1);
\draw   (4,1)   +(-0.05,-0.1)   --  +(-0.05,0.1)   +(0.05,-0.1)   -- +(0.05,0.1);

\draw   (2,3)   +(0,-0.1)  -- +(0,0.1)  +(-0.1,-0.1) -- +(-0.1,0.1)  +(0.1,-0.1) --  +(0.1,0.1);
\draw   (0,-1)   +(0,-0.1)  -- +(0,0.1)  +(-0.1,-0.1) -- +(-0.1,0.1)  +(0.1,-0.1) --  +(0.1,0.1);

\draw  (-2,-1) +(0,-0.1)  -- +(0,0.1) ;
\draw  (-1,0) +(0,-0.1)  -- +(0,0.1) ;

\draw  (5,0) +(-0.15,-0.1)   --  +(-0.15,0.1)   +(-0.05,-0.1)   -- +(-0.05,0.1)
+(0.05,-0.1)   --  +(0.05,0.1)   +(0.15,-0.1)   -- +(0.15,0.1);
\draw  (6,1) +(-0.15,-0.1)   --  +(-0.15,0.1)   +(-0.05,-0.1)   -- +(-0.05,0.1)
+(0.05,-0.1)   --  +(0.05,0.1)   +(0.15,-0.1)   -- +(0.15,0.1);

\draw   (1,1.5)  node[left]  {$\scriptstyle   \alpha_{1,1}$}  (2.25,2)
node[above]    {$\scriptstyle   \beta_{1,1}$}    (1,-2.5)   node[left]
{$\scriptstyle  \alpha_{1,2}$}  (2.25,-2)  node[above]  {$\scriptstyle
  \beta_{1,2}$} (4.5,0.75)  node[right] {$\scriptstyle \alpha_{2,1}$} (7,1)
node[above] {$\scriptstyle \beta_{2,1}$}
(1.5,-0.25)  node[right] {$\scriptstyle \alpha_{2,2}$} (-2,0)
node[above] {$\scriptstyle \beta_{2,2}$};

\fill[red] (2,0) circle (1.5pt);
\end{tikzpicture}
}
\end{minipage}
\hskip 40mm
\begin{minipage}[l]{0.23\linewidth}
\centering
\begin{tikzpicture}[scale=0.35]
 \draw (0,6) -- (1,8) -- (1,12) -- (5,12) -- (5,8) -- (7,8) -- (6,6) -- (8,6) -- (7,4) -- (7,0) -- (3,0) --  (3,4) -- (1,4) -- (2,6) -- cycle;
 \foreach \x in {(1,8),(2,6),(3,4)} \draw \x -- +(4,0);

 \draw[thin, <->, >=angle 45]  (5.2,8) -- (5.2,12); \draw (5.2,10) node[right] {\tiny $\lambda/2$};
 \draw[thin, <->, >=angle 45]  (1,10) -- (5,10); \draw (3,10) node[above] {\tiny $\lambda/2$};
 \draw[thin, <->, >=angle 45]  (-0.2,6) -- (-0.2,8); \draw (-0.2,7) node[left] {\tiny $h/2$};
 \draw[thin, <->, >=angle 45]  (0,8) -- (1,8); \draw (0.5,8) node[above] {\tiny $t/2$};
 \draw[thin, <->, >=angle 45]  (0.5,7) --( 6.5,7); \draw (3.5,7) node[above] {\tiny $w/2$};

 \foreach \x in {(1,6),(2,12)}{\draw \x +(0,0.1) -- +(0,-0.1);}
 \foreach \x in {(2,4), (4,12)} \draw \x +(-0.1,0.1) -- +(-0.1,-0.1) +(0.1,0.1) -- +(0.1,-0.1);
 \foreach \x in {(4,0),(6,8)} \draw \x +(-0.15,0.1) -- +(-0.15,-0.1) +(0,0.1) -- +(0,-0.1) +(0.15,0.1) -- +(0.15,-0.1);
 \foreach \x in {(6,0),(7,6)} \draw \x +(-0.1,0.1) -- +(0.1,-0.1) +(-0.1,-0.1) -- +(0.1,0.1);

  \foreach \x in {(0,6), (1,4), (1,8), (1,12), (2,6), (3,0), (3,4), (3,12), (5,0), (5,8), (5,12), (6,6), (7,0), (7,4), (7,8), (8,6)} \filldraw[fill=white] \x circle (3pt);
  \fill[red] (4,6) circle (2pt);
\end{tikzpicture}
\end{minipage}
\caption{Basis
 $\{\alpha_{i,j},\beta_{i,j}\}_{i,j=1,2}$ of $H_1(X,\Z)$ associated with  cylinder decompositions
  of  Model (left)  $A$ and Model $B$ (right). For $i=1,2$, setting $\alpha_{i}:=\alpha_{i,1}+\alpha_{i,2}$ and
$\beta_i:=\beta_{i,1}+\beta_{i,2}$, then $\{\alpha_1,\beta_1,\alpha_2,\beta_2\}$ is a  symplectic  basis     of
  $H_{1}(X,\Z)^{-}$.  }
\label{fig:H6:ModAB}
\end{figure}

\begin{proof}[Proof of Proposition~\ref{prop:normalize:A}]
The proof follows the same lines as the proof of~\cite[Prop. 4.5]{Lanneau:Manh:H4}. The only difference is in the intersection
form on $H_1(X,\Z)^-$. In this case, the intersection form (in the basis $\{\alpha_1,\beta_1,\alpha_2,\beta_2\}$)
is $\DS{\left( \begin{smallmatrix} 2J & 0 \\ 0 & 2J \\ \end{smallmatrix} \right)}$. All the computations are
straightforward.
\end{proof}

\begin{Remark}\label{rmk:cond:prot:A:B}
The decomposition is of Model $A$ if and only if
$$
\lambda < w/2 \iff 2(e+2h) < w \iff (e+4h)^2< D,
$$
and of Model B if and only if
$$
w/2\ < \lambda < w \iff  e+h < w < 2(e+2h) \iff (e+2h)^2 < D <(e+4h)^2
$$
\end{Remark}
For any discriminant  $D$, we denote by $\Pcal_D$ the set of $(w,h,t,e) \in \Z^4$  satisfying $(\Pcal_D)$.
Elements of $\Pcal_D$ are called {\em prototypes}.
We also denote by $\Pcal^A_D,\Pcal^B_D$ the set of prototypes of Model $A$ and $B$, that is
$$
\Pcal^A_D:=\{(w,h,t,e) \in \Pcal_D, \, \lambda < w/2\}.
$$
$$
\Pcal^B_D:=\{(w,h,t,e) \in \Pcal_D, \, w/2< \lambda < w\}.
$$
The surface constructed from  a prototype $(w,h,t,e) \in \Pcal_D$ will be denoted by $X_D(w,h,t,e)$.

\subsection{Prototypes of model $A$}
We show that for any discriminant $D\neq 5$, any surface in $\Omega E_D(6)$ admits a decomposition in Model $A$
(compare with~\cite[Prop. 4.7]{Lanneau:Manh:H4}).
\begin{Proposition}
\label{prop:H6:noModA}
Let $(X,\omega) \in \Omega E_D(6)$ that does not admit any decomposition in model $A$.
Then, up to the action of $\mathrm{GL}^{+}(2,\R)$, $(X,\omega)$ is the surface presented in
Figure~\ref{fig:modelB:H6} (on the right).
In particular, the order $\mathcal{O}_D$ is isomorphic to $\Z[x]/(x^2+x-1)$ and $D=5$.
\end{Proposition}

\begin{proof}[Proof of Proposition~\ref{prop:H6:noModA}]
Since $(X,\omega)$ is a Veech surface, we can assume that $(X,\omega)$ is horizontally periodic.
By assumption, the cylinder decomposition in the horizontal direction is in Model B.
Using $\GL^+(2,\R)$-action, we can normalize $(X,\omega)$ the larger cylinders are represented by two unit squares.
Let $0<x<1, 0<y, 0\leq t < x$  be the width, height, and twist of the smaller ones (see Figure~\ref{fig:modelB:H6}).

\begin{figure}[htb]
\begin{minipage}[t]{0.4\linewidth}
\begin{tikzpicture}[scale=0.3]
\draw (0,0) -- (2,0) -- (2,-4) -- (3,-4) -- (1.5,-6) -- (4.5,-6) -- (6,-4) -- (6,0) -- (4,0) -- (4,4) -- (3,4) -- (4.5,6) -- (1.5,6) -- (0,4) -- cycle;
\draw (0,4) -- (4,4) (2,0) -- (4,0) (3,-4) -- (6,-4);
\draw[dashed, thin] (0,6) -- (1.5,6) (0,0) -- (1.5,6) (2,0) -- (3.5,6) (2,0) -- (0.5,6);

\draw[thin, >=angle 45, <->] (-0.25,0) -- (-0.25,4); \draw (-0.25,2) node[left] {\tiny $1$};
\draw[thin, >=angle 45, <->] (-0.25,4) -- (-0.25,6); \draw (-0.25,5) node[left] {\tiny $y$};
\draw[thin, >=angle 45, <->] (0,6.25) -- (1.5,6.25); \draw (0.75,6.25) node[above] {\tiny $t$};
\draw[thin, >=angle 45, <->] (1.5,6.25) -- (4.5,6.25); \draw (3.25,6.25) node[above] {\tiny $x$};
\draw[thin, >=angle 45, <->] (0,-0.25) -- (2,-0.25);  \draw (1,-0.25) node[below] {\tiny $2x-1$};
\draw[thin, >=angle 45, <->] (2,-2) -- (6,-2); \draw (4,-2) node[above] {\tiny $1$};
\foreach \x in {(0,0), (2,0), (2,-4), (3,-4), (1.5,-6), (2.5,-6), (4.5,-6), (6,-4), (6,0), (4,0), (4,4), (3,4), (4.5,6), (3.5,6), (1.5,6), (0.5,6), (0,4)} \filldraw[fill=white] \x circle (2pt);
\end{tikzpicture}
\end{minipage}
\begin{minipage}[t]{0.4\linewidth}
\begin{tikzpicture}[scale=0.3]
\fill[green!20] (0,6) -- (4,0) -- (intersection of 3,4 -- 5,1 and 4,4 -- 4,0) -- (intersection of 3,4 -- 1,7 and 0,6 -- 2,6) -- cycle;
\fill[green!20] (0,0) -- (intersection of -1,4 -- 1,1 and 0,0 -- 0,4) -- (intersection of 1,1 -- 3,-2 and 0,0 -- 2,0 ) -- cycle;
\draw[thin, >=angle 45, <->] (-0.25,0) -- (-0.25,4); \draw (-0.25,2) node[left] {\tiny $1$};
\draw[thin, >=angle 45, <->] (2,-2) -- (6,-2); \draw (4,-2) node[above] {\tiny $1$};
\draw[thin, >=angle 45, <->] (-0.25,4) -- (-0.25,6); \draw (-0.25,5) node[left] {\tiny $\frac{-1+\sqrt{5}}{2}$};
\draw[thin, >=angle 45, <->] (-0.25,6.25) -- (3,6.25); \draw (1.6,6.25) node[above] {\tiny $\frac{-1+\sqrt{5}}{2}$};

\draw (0,0) -- (2,0) -- (2,-4) -- (3,-4) -- (3,-6) -- (6,-6) -- (6,0) -- (4,0) -- (4,4) -- (3,4) -- (3,6) -- (0,6) -- cycle;

\draw (0,4) -- (3,4) (2,0) -- (4,0) (3,-4) -- (6,-4);

\draw[dashed, thin] (4,0) -- (0,6) (intersection of 3,4 -- 1,7 and 0,6 -- 2,6) -- (intersection of 3,4 -- 5,1 and 4,4 -- 4,0)
(intersection of -1,4 -- 1,1 and 0,0 -- 0,4) -- (intersection of 1,1 -- 3,-2 and 0,0 -- 2,0 );
\foreach \x in {(0,0), (2,0), (2,-4), (3,-4), (3,-6), (4,-6), (6,-6), (6,-4), (6,0), (4,0), (4,4), (3,4), (3,6), (2,6), (0,6), (0,4)} \filldraw[fill=white] \x circle (3pt);
\end{tikzpicture}
\end{minipage}
\caption{Model $B$: cylinders in directions $v_1, v_2$ (left), and  $v_3$ (right).}
\label{fig:modelB:H6}
\end{figure}
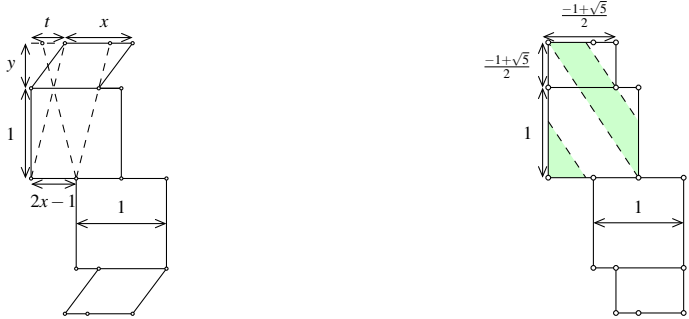
\noindent We first show $t=0 \mod x$. Assume $t>0$. There exists a cylinder in direction $v_1=\frac{y+1}{t}$.
Since $t>0$ this cylinder is not simple only when
\begin{equation}\label{cond:t:1}
\frac{1}{1-x}=\frac{y+1}{t}, \qquad \textrm{ or equivalently} \qquad t=(1-x)(1+y).
\end{equation}
Now, $t-(1-x)= (1-x)y >0$ implies that there exists a cylinder in direction $v_2=\frac{1+y}{x-t}$.
This cylinder is not simple only when $v_2$ is the vertical direction, which implies $t=x$.
\medskip

\noindent Since $t=0 \mod x$, condition~\eqref{cond:t:1} reads
$$\left\{
\begin{array}{ccl}
1+y & = & \frac{x}{1-x}, \\
 y &= & \frac{2x-1}{1-x}.
\end{array}
\right. \Rightarrow \frac{y}{y+1}=\frac{2x-1}{x}=2-\frac{1}{x}.
$$
It follows that $\frac{y}{y+1}<x$. Hence there exists a cylinder in direction $v_3=-(y+1)$.
This cylinder is not simple only if
$$
\frac{y}{y+1}=\frac{2x-1}{x}=1-x \Rightarrow x^2+x-1=0
$$
Solving above equation gives $x=y=\frac{-1+\sqrt{5}}{2}$ proving the proposition.
\end{proof}

\subsection{Butterfly moves}\label{sec:butterfly}
Let $(X,\omega):=X_D(w,h,t,e)$ be a prototypical  surface in $\Omega E_D(6)$ associated to a prototype $(w,h,t,e) \in \Pcal_D^A$.
We denote horizontal cylinders  of $X$ by $C_{i,j}, \ i,j \in \{1,2\}$, where $C_{i,1}$ and $C_{i,2}$ are exchanged by the
Prym involution, and $C_{1,j}$ is a simple cylinder.

Let $C'_{1}$ (resp. $C'_2$) be a simple cylinder contained in the closure of $C_{2,1}$ (resp. in the closure of $C_{2,2}$)
such that $C'_1$ and $C'_2$ are exchanged by the Prym involution $\tau$.
Note that $C'_1$ and $C'_2$ are disjoint from $C_{1,1}\cup C_{1,2}$.

Let $\alpha'_{1,j}$ be the element in $H_1(X,\Z)$ represented by the core curves of $C'_j$,
the orientation of the core curves are chosen such that $\tau(\alpha'_{1,1})=-\alpha'_{1,2}$.

We can write $\alpha'_{1,j}=p\alpha_{2,j}+q\beta_{2,j} \in H_1(X,\Z)$,
with $p\in \Z$, $q \in \Z\setminus\{0\}$ such that $\gcd(p,q)=1$.
Moreover, we can choose the orientation of $\alpha'_{1,j}$ such that $q >0$.
The following lemma gives a necessary and sufficient condition on $(p,q)$ for the existence of $C'_j$.
Its proof follows the same lines as \cite[Lem.7.2]{Lanneau:Manh:H4}.

\begin{Lemma}[Admissibility condition]
\label{lm:BM:admis:cond}
The simple cylinders $C'_j, \ j=1,2,$ exist if and only if
$$
0 < \lambda q < w/2 \Leftrightarrow (e+4qh)^2 < D.
$$
\end{Lemma}

Since $C'_j$ are simple cylinders, the surface $X$ admits a cylinder decomposition of Model A in the direction of $C'_j$.
Let $(w',h',t',e')$ be the prototype in $\Pcal_D^A$ associated to this cylinder  decomposition.
For our purpose, we will give a sketch of proof of the following proposition
(which parallels the proof of \cite[Prop.7.5,7.6]{Lanneau:Manh:H4}).

\begin{Proposition}\label{prop:BMpq:par:change}
Let $\Bc=(\alpha_1,\beta_1,\alpha_2,\beta_2)$ and $\Bc'=(\alpha'_1,\beta'_1,\alpha'_2,\beta'_2)$
denote the symplectic bases of $H_1^-(X,\Z)$ associated to $(w,h,t,e)$ and $(w',h',t',e')$ respectively.
Then the transition matrix $M$ of the basis change  from $\Bc$ to $\Bc'$  satisfies $M=M_1\cdot M_2\cdot M_3$,
where $M_1 \in \left(\begin{smallmatrix} \Id_2 & 0 \\ 0 & \SL(2,\Z) \end{smallmatrix}\right), \;
M_2= \left( \begin{smallmatrix}
 0 & 2 & 1 & 0 \\
 0 & 0 & 0 & 1 \\
 1 & 0 & 0 & 2 \\
 0 & 1 & 0 & 0 \\
\end{smallmatrix}
\right), \; M_3 \in \left(
\begin{smallmatrix}
 \left(\begin{smallmatrix} 1 & * \\ 0 & 1 \end{smallmatrix}\right) & 0 \\
 0 &  \SL(2,\Z) \\
\end{smallmatrix}
\right)$.  As a consequence, the new prototype $(w',h',t',e')$ satisfies
$$
\left\{\begin{array}{ccl}
e' & = & -e-4qh,\\
h' & = & \gcd(-qh,pw+qt)\\
\end{array}\right.
$$
\end{Proposition}
\begin{proof}
Let $\eta'_{1,1}, \eta'_{1,2}$ be two saddle connections contained in $C'_1$ and $C'_2$ respectively such that
$\eta'_{1,2}=-\tau(\eta'_{1,1})$, where $\tau$ is the Prym involution (see Figure~\ref{fig:BM:pq:basis:change}).
Set $\alpha'_1=\alpha'_{1,1}+\alpha'_{1,2}, \quad \eta'_1=\eta'_{1,1}+\eta'_{1,2}$.
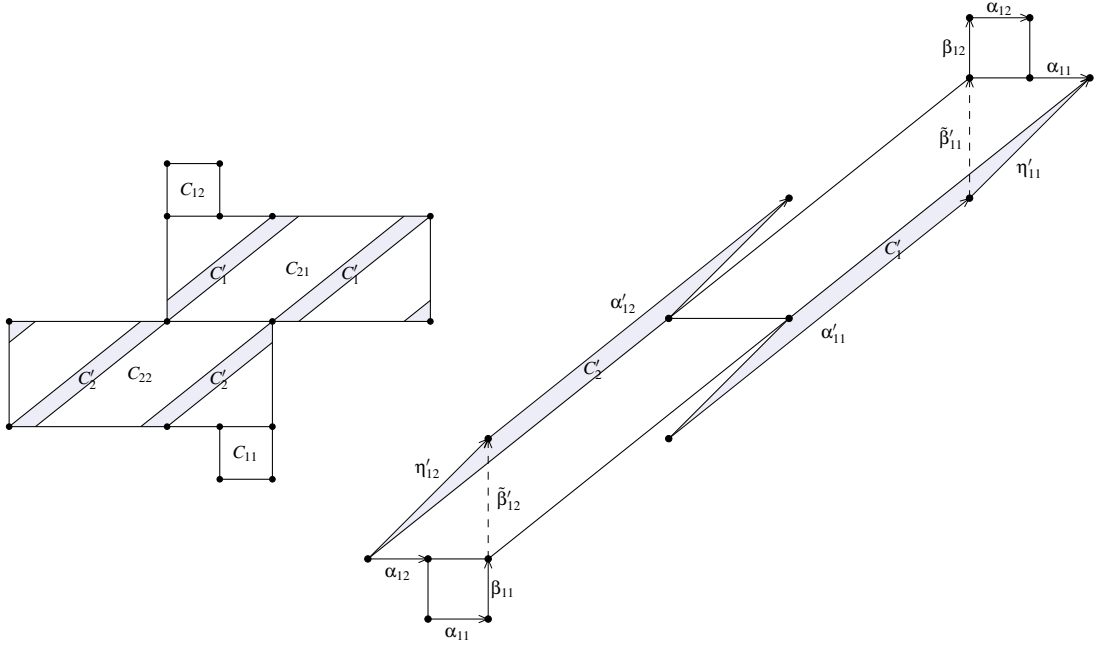
\begin{figure}[htb]
\begin{minipage}[l]{0.3\linewidth}
\centering
\begin{tikzpicture}[scale=0.35]
\fill[blue!70!yellow!10] (6,6) -- (11,10) -- (10,10) -- (6,6.8) -- cycle;
\fill[blue!70!yellow!10] (10,6) -- (15,10) -- (16,10) -- (11,6) --cycle;
\fill[blue!70!yellow!10] (15,6) -- (16,6) -- (16,6.8) -- cycle;
\fill[blue!70!yellow!10] (0,6) -- (0,5.2) -- (1,6) -- cycle;
\fill[blue!70!yellow!10] (0,2) -- (1,2) -- (6,6) -- (5,6) -- cycle;
\fill[blue!70!yellow!10] (5,2) -- (6,2) -- (10,5.2) -- (10,6) -- cycle;

\draw (0,6) -- (0,2) -- (8,2) -- (8,0) -- (10,0) -- (10,6) -- (16,6) -- (16,10) -- (8,10) -- (8,12)  -- (6,12) -- (6,6) -- cycle;
\draw (0,5.2) -- (1,6)  (0,2) -- (5,6) (6,6.8) -- (10,10) (1,2) -- (11,10)  (5,2) -- (15,10) (6,2) -- (10,5.2)  (11,6)  -- (16,10)  (15,6) -- (16,6.8);
\draw (6,10) -- (8,10)  (6,6) -- (10,6) (8,2) -- (10,2);

\foreach \x in {(0,6),(0,2),(6,12),(6,10),(6,6),(6,2),(8,12), (8,10),(8,2),(8,0),(10,10),(10,6),(10,2),(10,0),(16,10),(16,6)} \filldraw \x circle (3pt);

\draw (7,11) node {\tiny $C_{12}$} (11,8) node {\tiny $C_{21}$} (5,4) node {\tiny $C_{22}$} (9,1) node {\tiny $C_{11}$};
\draw (8,7.8) node {\tiny $C'_1$} (13,7.8) node {\tiny $C'_1$} (3,3.8) node {\tiny $C'_2$} (8,3.8) node {\tiny $C'_2$};
\end{tikzpicture}
\end{minipage}
\begin{minipage}[l]{0.6\linewidth}
\centering
\begin{tikzpicture}[scale=0.4]
\fill[blue!70!yellow!10] (0,0) -- (10,8) -- (14,12) -- (4,4) -- cycle;
\fill[blue!70!yellow!10] (10,4) -- (20,12) -- (24,16) -- (14,8) -- cycle;

\draw[very thin, ->, >=angle 45] (20,18) -- (22,18);
\draw[very thin, ->, >=angle 45] (22,16) -- (24,16);
\draw[very thin, ->, >=angle 45] (10,4) -- (20,12);
\draw[very thin, ->, >=angle 45] (20,12) -- (24,16);
\draw[very thin, ->, >=angle 45] (20,16) -- (20,18);
\draw[very thin, dashed, ->, >=angle 45] (20,12) -- (20,16);

\draw[very thin, ->, >=angle 45] (0,0) -- (2,0);
\draw[very thin, ->, >=angle 45] (2,-2) -- (4,-2);
\draw[very thin, ->, >=angle 45] (0,0) -- (4,4);
\draw[very thin, ->, >=angle 45] (4,4) -- (14,12);
\draw[very thin, ->, >=angle 45] (4,-2) -- (4,0);
\draw[very thin, dashed, ->, >=angle 45] (4,0) -- (4,4);

\draw (0,0) -- (20,16)  (22,18) -- (22,16)  (24,16) -- (4,0)   (2,-2) -- (2,0);
\draw (2,0) -- (4,0) (10,4) -- (14,8) -- (10,8) -- (14,12)  (20,16) -- (22,16);

\foreach \x in {(0,0),(2,0),(2,-2),(4,4),(4,0),(4,-2),(10,8),(10,4),(14,12),(14,8),(20,18),(20,16),(20,12),(22,18),(22,16),(24,16)} \filldraw \x circle (3pt);

\draw (21,18.3) node {\tiny $\alpha_{12}$}  (23,16.3) node {\tiny $\alpha_{11}$};
\draw (19.5,17) node {\tiny $\beta_{12}$};
\draw (1,-0.5)  node {\tiny $\alpha_{12}$} (3,-2.5) node {\tiny $\alpha_{11}$};
\draw (4.5,-1) node {\tiny $\beta_{11}$};
 \draw (22,13) node {\tiny $\eta'_{11}$} (2,3) node {\tiny $\eta'_{12}$};
 \draw (15.5,7.5) node {\tiny $\alpha'_{11}$} (8.5,8.5) node {\tiny $\alpha'_{12}$} ;
 \draw (19.4,14)  node {\tiny $\tilde{\beta}'_{11}$} (4.7,2) node {\tiny $\tilde{\beta}'_{12}$};
 \draw (17.5,10.3) node {\tiny $C'_1$} (7.5,6.3) node {\tiny $C'_2$} ;
\end{tikzpicture}
\end{minipage}
\caption{Switching periodic directions: symplectic basis change.}
\label{fig:BM:pq:basis:change}
\end{figure}


\noindent \underline{Step 1:} set $\tilde{\beta}'_{1,j}=\eta'_{1,j}-\alpha_{1,1}-\alpha_{1,2} \in H_1(X,\Z)$
(see Figure~\ref{fig:BM:pq:basis:change}),  and $\tilde{\beta}'_1=\tilde{\beta}'_{1,1}+\tilde{\beta}'_{1,2}$.
We have, $(\alpha'_1,\tilde{\beta}'_1)=(\alpha_2,\beta_2)\cdot \left( \begin{smallmatrix} p & r \\ q & s \end{smallmatrix} \right)$,
where $\left( \begin{smallmatrix} p & r \\ q & s \end{smallmatrix} \right) \in \SL(2,\Z)$.
Therefore, $(\alpha_1,\beta_1,\alpha'_1,\tilde{\beta}'_1)$ is a symplectic basis of $H_1^-(X,\Z)$,
and $(\alpha_1,\beta_1,\alpha'_1,\tilde{\beta}'_1)=(\alpha_1,\beta_1,\alpha_2,\beta_2)\cdot M_1$,
where $M_1=\left(\begin{smallmatrix} \Id_2 & 0 \\ 0 & \left(  \begin{smallmatrix} p & r \\ q & s \end{smallmatrix} \right)
\end{smallmatrix}\right)$.

\noindent \underline{Step 2:} set
$$
\left\{ \begin{array}{lcl}
\tilde{\alpha}'_{2,j}  =  \alpha_{1,j}, \ j=1,2  & \Rightarrow & \tilde{\alpha}'_2:=\tilde{\alpha}'_{2,1}+\tilde{\alpha}'_{2,2}=\alpha_1\\
\tilde{\beta}'_{2,j} = \alpha'_{1,1}+\alpha'_{1,2} +\beta_{1,j} \ j=1,2 & \Rightarrow & \tilde{\beta}'_2:=\tilde{\beta}'_{2,1}+\tilde{\beta}'_{2,2}=\beta_1+2\alpha'_1.
\end{array}
\right.
$$
Recall that $\eta'_1=\eta'_{1,1}+\eta'_{1,2}=\tilde{\beta}'_2+2\alpha_1$.
Thus $(\alpha'_1,\eta'_1,\tilde{\alpha}'_2,\tilde{\beta}'_2)$ is a symplectic basis of $H_1(X,\Z)^-$,
and $(\alpha'_1,\eta'_1,\tilde{\alpha}'_2,\tilde{\beta}'_2)=(\alpha_1,\beta_2,\alpha'_1,\tilde{\beta}'_1)\cdot M_2$,
where $M_2= \left(
\begin{smallmatrix}
 0 & 2 & 1 & 0 \\
 0 & 0 & 0 & 1 \\
 1 & 0 & 0 & 2 \\
 0 & 1 & 0 & 0 \\
\end{smallmatrix}
\right)$.

\noindent \underline{Step 3:} the complement of $C'_1\cup C'_2$ in $X$ is the union of two cylinders $C''_1$ and $C''_2$ in the same direction.
Let $\alpha'_{2,j}$ be a core curve of  $C''_j$, and $\eta'_{2,j}$ a saddle connection in $C''_j$ that crosses $\alpha'_{2,j}$ once.
Set $\alpha'_2:=\alpha'_{2,1}+\alpha'_{2,2}, \eta'_2:=\eta'_{2,1}+\eta'_{2,2}$, then $(\alpha'_2,\eta'_2)=(\tilde{\alpha}'_2,\tilde{\beta}'_2)\cdot A$,
with $A \in \SL(2,\Z)$.

We now observe that  the symplectic basis $\Bc'$ of $H_1^-(X,\Z)$ adapted to the cylinder decomposition in the direction of $C'_j$ must be
$(\alpha'_1,\beta'_1,\alpha'_2,\beta'_2)$, where $\beta'_i$ is obtained from $\eta'_i$ by some Dehn twist. Therefore,
$(\alpha'_1,\beta'_1,\alpha'_2,\beta'_2)=(\alpha'_1,\eta'_1,\tilde{\alpha}'_2,\tilde{\beta}'_2)\cdot M_3$,
where $M_3 \in \left(
\begin{smallmatrix}
 \left(\begin{smallmatrix} 1 & * \\ 0 & 1 \end{smallmatrix}\right) & 0 \\
 0 &  \SL(2,\Z) \\
\end{smallmatrix}
\right)$, and the first assertion follows.

Let $T$ be the generator of $\Ord_D$ associated to the prototype $(w,h,t,e)$.
Recall that the matrix of $T$ in the basis $\Bc$ is given by
$T=\left(
\begin{smallmatrix}
 e & 0 & w & t \\
 0 & e & 0 & h \\
 h & -t & 0 & 0 \\
 0 & w  & 0 & 0 \\
\end{smallmatrix}
\right)$.
Let $T_2$ and $T_3$ be the matrices of $T$ in the bases $(\alpha'_1,\eta'_1,\tilde{\alpha}'_2,\tilde{\beta}'_2)$
and $(\alpha'_1,\beta'_1,\alpha'_2,\beta'_2)$ respectively.
A direct computation shows
$$
T_2=M_2^{-1}\cdot M_1^{-1}\cdot T \cdot M_1\cdot M_2=
\left(
\begin{smallmatrix}
 -2qh & 0 & a & b \\
 0 & -2qh & c & d \\
 d & -b & 2qh+e &  0  \\
 -c & a & 0 & 2qh+e \\
\end{smallmatrix}
\right)
$$
where
$$
\left\{
\begin{array}{ccl}
 a & = & sh, \\
 b & = & -4qh-rw-st-2e,\\
 c & = & -qh,\\
 d & = & pw +qt.
\end{array}
\right.
$$
Hence
$$
T_3=M_3^{-1}\cdot T_2 \cdot M_3=\left(
\begin{array}{cc}
 -2qh\cdot\Id_2 & \left(\begin{smallmatrix} 1 & -n \\ 0 & 1  \end{smallmatrix}\right)\cdot\left( \begin{smallmatrix} a & b \\ c & d \end{smallmatrix}\right) \cdot A \\
 A^{-1}\cdot \left( \begin{smallmatrix} d & -b \\ -c & a \end{smallmatrix}\right)\cdot \left(\begin{smallmatrix} 1 & n \\ 0 & 1  \end{smallmatrix}\right) & (2qh+e)\cdot \Id_2\\
\end{array}
\right), \; \text{ with } n\in \Z, \, A \in \SL(2,\Z).
$$
Consider now the generator $T'$ associated to the cylinder decomposition in the direction of $C'_1$.
The matrix of $T'$ in the basis $\Bc'$ is given by
$T'=\left(
\begin{smallmatrix}
 e' & 0 & w' & t' \\
 0 & e' & 0 & h' \\
 h' & -t' & 0 & 0 \\
 0 & w'  & 0 & 0 \\
\end{smallmatrix}
\right)$ with $(w',h',t',e') \in \Pcal_D^A$.
Since $T$ and $T'$ are both generators of $\Ord_D$, we must have $T'=\pm T+f\Id_4$, with $f \in \Z$.
Comparing the matrices of $T$ and $T'$ in $\Bc'$, and using the admissibility condition
$0 < \lambda q < w/2 \Leftrightarrow \lambda-e-2qh >0$, we get
$$
\left\{
\begin{array}{ccl}
T' & = & T-(e+2qh), \\
e' & = & -e -4qh, \\
h' & = & \gcd(c,d)=\gcd(-qh,pw+qt).\\
\end{array}
\right.
$$
\end{proof}

 We will call the operation of passing from the cylinder decomposition in the horizontal direction to the cylinder decomposition in the direction of $C'_1$ a {\em Butterfly move}. If the pair of integers associated with the core curve of $C'_1$ is $(1,q), q \in \Z\setminus\{0\}$, we denote the corresponding Butterfly move   by $B_q$. If this pair of integer is $(0,1)$, then the corresponding Butterfly move is denoted by $B_\infty$.
Note that the Butterfly moves preserve the type of the decomposition, thus they induce transformations on the set of prototypes  $\Pcal^A_D$.

By the same arguments as \cite[Lem.7.2]{Lanneau:Manh:H4} and \cite[Prop.7.5, Prop.7.6]{Lanneau:Manh:H4}
(see also \cite[Th.7.2, Th.7.3]{Mc4}), we can prove

\begin{Proposition}
\label{prop:BM:transform:rule}
The Butterfly move $B_\infty$ is always realizable.
For $q\in \N$, the Butterfly move $B_q$ is realizable on the prototypical surface $X_D(w,h,t,e)$  if we have
\begin{equation*}
0 < \lambda q < \frac{w}{2} \Leftrightarrow (e+4qh)^2<D
\end{equation*}

The actions of the Butterfly moves on $\Pcal^A_D$ are given by
\begin{enumerate}
\item If $q\in \N$ then $B_q(w,h,t,e)=(w',h',t',e')$ where
$$\left\{\begin{array}{ccl}
e' & = & -e-4qh,\\
h' & = & \gcd(qh,w+qt)\\
\end{array}\right. $$
\item If $q=\infty$ then $B_\infty(w,h,t,e)=(w',h',t',e')$ where
$$\left\{\begin{array}{ccl}
e' & = & -e-4h,\\
h' & = & \gcd(t,h)\\
\end{array}\right. $$
\end{enumerate}
\end{Proposition}

Lemma~\ref{lm:4cyl:dec:exist} and Proposition~\ref{prop:H6:noModA} imply the following
\begin{Theorem}
\label{theo:onto:map}
Let  $D$ be  a fixed  positive integer.  If $D\not =  5$ then there is an onto map  from $\Pcal^A_D$ on
the components of $\Omega E_{D}(6)$. \medskip

\noindent Let $\sim$ be the equivalence relation on $\Pcal^A_D$ that is generated by
the Butterfly moves $B_q$, that is $p\sim p'$ if and only if there is a sequence of Butterfly moves that send $p$ to $p'$.
Then we have
$$
 \#\   \{\textrm{Components   of }   \Omega    E_D(6)\}   \leq
\#\ \left(\Pcal^A_D/\sim\right).
$$
\end{Theorem}

An equivalence class  of the equivalence  relation generated  by  the
Butterfly moves will be called a  {\it  component} of $\Pcal^A_D$.

\subsection{Reduced prototypes and almost reduced prototypes}\label{sec:red:prototypes}
A {\em reduced  prototype} in $\Pcal^A_D$ is a  prototype $(w,h,t,e) \in \Pcal^A_D$ where $h=1$, and $t=0$.
The set of reduced prototypes of a discriminant $D$ is denoted by $\Sc^1_D$.

When $D\equiv 1 \mod 8$, we will also use the set
$$
\Sc_D^2  = \{ (w,h,t,e) \in \Pcal^A_D, \, h = 2, \, t=0, \, w \textrm{ is even}\}.
$$
Elements of $\Sc_D^2$ will be called {\em almost-reduced  prototypes}.
We close this section by the following
\begin{Lemma}\label{lm:reduced}\hfill
\begin{enumerate}
\item If $D\not\equiv 1 \mod 8$ then any element of $\Pcal^A_D$ is equivalent to an element of $\Sc^1_D$.

\item If $D\equiv 1 \mod$, then any element of $\Pcal_D^A$ is equivalent to either an element of $\Sc_D^1$ or an element of $\Sc_D^2$.
\end{enumerate}

\end{Lemma}
\begin{proof}
Let $p_0=(w_0,h_0,t_0,e_0)$ be an element in the equivalence class of $p$ such that $h_0$ is minimal.
Since the Butterfly move $B_\infty$ is always admissible, we must have $h_0 \leq \gcd(t_0,h_0)$.
But $t_0< h_0$, therefore $t_0=0$.
Applying the Butterfly move $B_1$ (which is always admissible), we get $h_0 \leq  \gcd(h_0,w_0)$,
which implies that $h_0\mid w_0$.

Let $(w'_0,h'_0,t'_0,e'_0)=B_1(w_0,h_0,t_0,e_0)$. We have $h'_0=h_0$, and $e'_0=-e_0-4h_0$.
It follows that $w'_0=\frac{D-(e_0+4h_0)^2}{4h_0}=w_0-2e_0-4h_0$.
The same argument as above shows that we must have $h_0 \mid w'_0$, which implies $h_0\mid 2e_0$.

We first consider the case $D\not\equiv 1\mod 8$, which means that $d \equiv 0,4,5 \mod 8$.
If $D$ is even then so is $e_0$. If $h_0$ is also even then $2 \mid\gcd(w_0,h_0,e_0)$,
which is impossible since $\gcd(w_0,h_0,e_0)=1$ by the definition of prototype.
Thus $h_0$ must be odd. Since $h_0\mid 2e_0$, we draw that $h_0\mid e_0$.
Hence $h_0=\gcd(w_0,h_0,e_0)=1$, and $(w_0,h_0,t_0,e_0) \in \Sc^1_D$.

If $D\equiv 5\mod 8$, then since $e^2_0\equiv 1\mod 8$, we have $w_0h_0=\frac{D-e_0^2}{4}$ is odd,
which implies that $h_0$ is odd. The same argument as above shows that $h_0=1$ and $(w_0,h_0,t_0,e_0)\in \Sc^1_D$.

We now consider the case $D\equiv 1\mod 8$. If $h_0$ is odd, since $h_0 \mid 2e_0$, we must have $h_0\mid e_0$.
Hence $h_0=\gcd(w_0,h_0,e_0)=1$, and $p_0\in \Sc_D^1$.
If $h_0$ is even then $h_0/2 \mid e_0$, thus $h_0/2=\gcd(w_0,h_0,e_0)=1$.
Therefore, we have $h_0=2$. Since $2|w_0$, we have $p_0 \in \Sc^2_D$.
\end{proof}


\section{Components of $\Pcal^A_D$}
\subsection{Disconnectedness of $\Pcal^A_D$}.

The following lemmas show that $\Pcal^A_D$ have more than one component in general.
\begin{Lemma}\label{lm:no:conn:PAD:even}
If $D\geq 20$ is an even discriminant, that is $D\equiv 0 \mod 4$ then $\Pcal^A_D$ has at least two components.
\end{Lemma}
\begin{proof}
Let $p=(w,h,t,e) \in \Pcal^A_D$ be a prototype. Since $D- e^2=4wh$, $e$ must be even, that is $e \equiv 0,2 \mod 4$.
Assume that $p$ is mapped by some Butterfly move $B_q$ to another prototype $p'=(w',h',t',e')$. Then by Proposition~\ref{prop:BM:transform:rule},
we must have $e'\equiv e \mod 4$.
Thus,  $p_1=(\frac{D-4}{4}, 1,0,-2)$ and $p_2=(\frac{D}{4},1,0,0)$ cannot belong to the same equivalence class of $\sim$.
\end{proof}

\begin{Lemma}\label{lm:D1mod8:S1:no:connect:S2}
Let $D>9$ be a discriminant such that $D\equiv 1 \mod 8$.
Let $p_0=(w_0,h_0,t_0,e_0)$ be an element of $\Pcal^A_D$ such that $w_0\equiv h_0\equiv t_0 \equiv 0 \mod 2$.
If the prototype $p_1=(w_1,h_1,t_1,e_1) \in \Pcal_D^A$ satisfies
$\left(\begin{smallmatrix} w_1 & t_1 \\ 0 & h_1 \end{smallmatrix}\right)
\not\equiv \left(\begin{smallmatrix} 0 & 0 \\ 0 & 0 \end{smallmatrix}\right) \mod 2$,
then $p_1$ is not contained in the equivalence class of $p_0$.
\end{Lemma}
\begin{proof}
For any   element $p=(w,h,t,e)$  of $\Pcal_D^A$, let us denote by $T_p$ the generator of $\Ord_D$
associated to $p$. The matrix of $T_p$ in the basis of $H_1^-(X,\Z)$ adapted
to the corresponding cylinder decomposition is given by
$T_p=\left(
\begin{smallmatrix}
 e & 0 & w & t\\
 0 & e & 0 & h \\
 h & -t & 0 & 0 \\
 0 & w & 0 & 0 \\
\end{smallmatrix}
\right)$ (see Proposition~\ref{prop:normalize:A}). In particular, the matrix of the generator of $\Ord_D$ associated to $p_0$ satisfies
$T_{p_0} \equiv \left(\begin{smallmatrix} \Id_2 & 0 \\ 0 & 0 \end{smallmatrix}\right) \mod 2$.

Let  $p'=(w',e',h',t') \in \Pcal_D^A$ be the prototype obtained from $p$ by an admissible Butterfly move
$B_{(m,n)}$. We claim that the matrix $T'_p$ of $T_p$ in the basis of $H_1^-(X,\Z)$ associated with $p'$ also satisfies
$T'_{p} \equiv \left(\begin{smallmatrix} \Id_2 & 0 \\ 0 & 0 \end{smallmatrix}\right) \mod 2$.
To see this, recall that by Proposition~\ref{prop:BMpq:par:change} the matrix of the basis change induced by the Butterfly move
is given by $M_1\cdot M_2\cdot M_3$, where
$$
M_1\in \left(
\begin{smallmatrix}
 \Id_2 & 0 \\
 0     & \SL(2,\Z)\\
\end{smallmatrix}
\right), \quad M_2=\left(
\begin{smallmatrix}
 0 & 2 & 1 & 0 \\
 0 & 0 & 0 & 1 \\
 1 & 0 & 0 & 2 \\
 0 & 1 & 0 & 0 \\
\end{smallmatrix}
\right), \quad M_3\in \left(
\begin{smallmatrix}
 \begin{smallmatrix} 1 & * \\ 0 & 1 \end{smallmatrix} & 0 \\
 0 & \SL(2,\Z)\\
\end{smallmatrix}
\right).
$$
Since $T'_p=(M_1\cdot M_2\cdot M_3)^{-1} \cdot T_p \cdot M_1 \cdot M_2 \cdot M_3$, it is easy to check that
$T'_p\equiv \left(\begin{smallmatrix} \Id_2 & 0 \\ 0 & 0 \end{smallmatrix}\right) \mod 2$.

Now, assume that $p_0$ can be connected to $p_1=(w_1,h_1,t_1,e_1)$ by a sequence of Butterfly moves.
Let $T'_{p_0}$ be the matrix of $T_{p_0}$ in the basis adapted to $p_1$.
The previous claim implies that $T'_{p_0}\equiv \left(\begin{smallmatrix} \Id_2 & 0 \\ 0 & 0 \end{smallmatrix}\right) \mod 2$.
Since $T'_{p_0}$ and $T_{p_1}$ are both generators of $\Ord_D$, we must have
$T_{p_1}=\pm T'_{p_0}+f\Id_4$, with $f\in \Z$.
But this is impossible since the  top right $2\times 2$ submatrix of $T_{p_1}$
is equal to $\left(\begin{smallmatrix} w_1 & t_1 \\ 0 & h_1 \end{smallmatrix}\right) \not\equiv 0 \mod2$,
while the same submatrix of $\pm T'_{p_0}+f\Id_4$ is equal to $0$ modulo $2$.
This contradiction allows us to conclude.
\end{proof}

As a consequence of Lemma~\ref{lm:D1mod8:S1:no:connect:S2}, we get

\begin{Corollary}\label{cor:D1mod8:S1:no:connect:S2}
If $D\equiv 1 \mod 8$, then an element of $\Sc_D^1$ is not equivalent to any element of $\Sc_D^2$.
\end{Corollary}

The following theorem shows that essentially, that is for $D$ large enough, $\Pcal^A_D$ does not have other components than the ones mentioned in Lemmas~\ref{lm:no:conn:PAD:even} and~\ref{lm:D1mod8:S1:no:connect:S2}.

\begin{Theorem}[Components of $\Pcal^A_D$]
\label{theo:H6:connect:PD}
Let $D\geq 4$  be a  discriminant. Assume that
$$
D \not \in \mathrm{Exc}_1:=\{4,5,8,9,12,16,17, 25, 33, 36,41, 49,52,68,84,100\}
$$
and
$$
D \not \in \mathrm{Exc}_2:=\{113,145, 153, 177, 209,265,313,481\}.
$$
Then the space ${\mathcal P}^A_{D}$ is non empty and has
\begin{enumerate}
\item one component if $D\equiv 5\mod 8$,

\item two components, $\{(w,h,t,e) \in {\mathcal P}^A_{D},\ e\equiv 0 \mod 4\}$ and $\{(w,h,t,e)\in {\mathcal P}^A_{D},\ e\equiv 2 \mod 4\}$, if $D\equiv 0,4 \mod 8$,
\item two components $\allowbreak \Pcal^{A_1}_D:=\{(w,h,t,e)\in {\mathcal P}^A_{D},\  (w,h,t) \not\equiv (0,0,0) \mod 2\}$ and $\Pcal^{A_2}_D:=\{(w,h,t,e) \in {\mathcal P}^A_{D},\ w\equiv h \equiv t \equiv 0 \mod 2\}$, if $D\equiv 1 \mod 8$.
\end{enumerate}
For $D\in \mathrm{Exc}_1$, we have
\begin{itemize}
\item If $D\in\{4,5,9\}$ then ${\mathcal P}^A_{D}$ is empty.
\item if $D\in \{8,12,16,17, 25, 33, 49\}$ then $\Pcal^A_D$ has only one component.
\item If $D \in \{36,41,52,68,84,100\}$, then ${\mathcal P}^A_{D}$ has three components.
\end{itemize}
For $D\in \mathrm{Exc}_2$, ${\mathcal P}^A_{D}$ has three components and ${\mathcal P}^{A_1}_{D}$ is connected.
\end{Theorem}

To prove this theorem, we use similar ideas to the proof of \cite[Th.8.6]{Lanneau:Manh:H4}. Even though there are some new technical difficulties related to the fact that when $D\equiv 1 \mod 8$, $\Pcal^A_D$ has two types of reduced prototypes $\Sc_D^1$ and $\Sc_D^2$, the same strategy actually allows us to get the desired conclusion. Theorem~\ref{theo:H6:connect:PD} is proved in details in Appendix~\ref{sec:prf:connect:PD}.

\section{Detecting prototypes using areas}
For our purpose, it is important to determine the prototype associated with a periodic direction.
While in principle it is possible to obtain all the parameters of the corresponding prototype,  the calculations could be quite complicated in practice. However, the following lemma shows that the parameter $e$ can be easily computed from the area of a cylinder in the direction under consideration.

\begin{Lemma}
\label{lm:cyl}
Let $(X,\omega)\in \Omega E_{D}(6)$ be a Prym eigenform with a semi-simple cylinder $\mathcal C$. Then
there is $g\in \mathrm{GL}^{+}(2,\R)$ such that $g\cdot(X,\omega)=X_D(w,h,t,e)$ and
$$
\mathrm{Area}(\mathcal C) = \mathrm{Area}(X,\omega) \cdot \frac{1}{2} \cdot \frac{\lambda}{\sqrt{D}}
$$
If $\mathcal{C}$ is simple then $(w,h,t,e) \in \Pcal^A_D$, and if $\mathcal{C}$ is strictly semi-simple $(w,h,t,e) \in \Pcal^B_D$.
In particular, if $(X,\omega)=X_D(w,h,t,e)\in \Omega E_{D}(6)$ is a Prym eigenform with a non-horizontal semi-simple cylinder $\mathcal C$, then
there is $g\in \mathrm{GL}^{+}(2,\R)$ such that $g\cdot(X,\omega)=X_D(w',h',t',e')$, with $(w',h',t',e')\in \Pcal_D$ and
$$
\mathrm{Area}(\mathcal C) =\frac{\lambda\cdot  \lambda'}{4}.
$$
\end{Lemma}
\begin{proof}[Proof of Lemma~\ref{lm:cyl}]
We only give the proof for the case $\mathcal C$ is a simple cylinder as the case $\mathcal{C}$ is strictly semi-simple follows from the same arguments.

Up to the action of $\mathrm{SL}(2,\R)$
one can assume that $\mathcal C$ is horizontal.
By Proposition~\ref{prop:normalize:A} there is an element $g= \left( \begin{smallmatrix} * & * \\ 0 & * \end{smallmatrix}\right) \in \GL^+(2,\R)$
such that $(Y,\eta)=g\cdot(X,\omega)=X_D(w,h,t,e)$.
In particular $g(\mathcal C)$ is a square of dimension $\frac1{2}\lambda$,
thus $\mathrm{Area}(g(\mathcal C))= \mathrm{Area}(\mathcal C)\cdot \det(g) =\frac1{4}\lambda^2$. On the other hand
$$
\Aa(Y,\eta)=\frac1{2} \cdot \left( \lambda^2 + wh\right) =\frac1{2} \cdot \left( \lambda^2 + \lambda^2-e\lambda\right)
=\frac{\lambda}{2} \cdot \left( 2\lambda-e\right) =\frac{\lambda}{2} \cdot \sqrt{D}
$$
Since $\Aa(Y,\eta)=\det(g)\cdot\Aa(X,\omega)$, the lemma follows.
\end{proof}

\begin{Proposition}
\label{prop:square:tiled}
A surface $(X,\omega)\in \Omega E_{D}(6)$ is square-tiled if and only if $D$ is a square, that is $D=d^{2}$.
Moreover if $(X,\omega)$ is primitive, made of $n$ squares, then $n=2d$.
\end{Proposition}
\begin{proof}
The first assertion is obvious.
Let us prove the second one.
Since $D=d^2$, we have $D \neq 5$, and Proposition~\ref{prop:H6:noModA} implies that $(X,\omega)$ belongs to the  $\GL^+(2,\R)$-orbit of a prototypical surface $X_D(w,h,t,e)$, with $p=(w,h,t,e) \in \Pcal^A_D$.
By Lemma~\ref{lm:reduced} we can suppose that $p$ is either reduced or almost-reduced.

Let us consider the case $p$ is reduced, that is  $p=(w,1,0,e)$.
Note that $X_D(w,1,0,e)$ is not a primitive square-tiled surface,
since we have $\lambda=\frac{e+d}{2}$, while $w=\frac{d+e}{2}\cdot\frac{d-e}{2}$.
Let $B=\left( \begin{matrix} 2/\lambda & 0 \\ 0 & 2\end{matrix}\right)$.
Then $(Y,\eta) = B\cdot X_D(w,1,0,e)$ is clearly primitive.
A simple computation shows
$$
\Aa(Y,\eta) = 2(\lambda + \lambda - e) = 2(d+e-e)=2d.
$$
which means that $(Y,\eta)$ is made of $2d$ squares.
Since
$$
\Z\oplus i\Z = \Lambda(B\cdot A\cdot (X,\omega))=B\cdot A\cdot\Lambda(X,\omega) = B\cdot A\cdot \Z\oplus i\Z
$$
the matrix $B\cdot A$ has determinant $1$. Hence $\Aa(X,\omega)=\Aa(Y,\eta)=2d$, that is $(X,\omega)$ is also made of $2d$ squares.

Assume now that $p$ is almost-reduced, that is $p=(w,2,0,e)$, where $w$ is even and $e$ is odd.
In this case $\Aa(X_D(w,2,0,e))=d\frac{e+d}{4}$.
To get a primitive square-tiled surface, we have to rescale $X_D(w,2,0,e)$ either by $\left(\begin{smallmatrix}\frac{4}{e+d} & 0 \\ 0 & 2 \end{smallmatrix}\right)$ if $\frac{e+d}{2}$ is odd, or by $\left(\begin{smallmatrix}\frac{8}{e+d} & 0 \\ 0 & 1 \end{smallmatrix}\right)$ if $\frac{e+d}{2}$ is even (which is equivalent to $\frac{d-e}{2}$ is odd).
In both cases, the resulting surface consists of exactly $2d$ squares.
\end{proof}
This proposition allows us to reformulate Lemma~\ref{lm:cyl} in the case $D$ is a square as follows
\begin{Corollary}
\label{cor:cylinder:lambda}
Let $(X,\omega)\in \Omega E_{D}(6)$ be a square-tiled surface with $D=d^2$.
Let $\Cc$ be a simple cylinder on $X$, and $(w,h,t,e)$
be the prototype associated to the cylinder decomposition
in the direction of $\Cc$.
Then there is $g\in \mathrm{GL}^{+}(2,\R)$ such that
$g\cdot(X,\omega)$ is a primitive square-tiled surface and
$$
\mathrm{Area}(g(\mathcal C)) = \lambda=\frac{d+e}{2}.
$$
\end{Corollary}
\section{Switching Model B to Model A}
To prove Theorem~\ref{thm:H6:eig:comp}, assuming that $D>5$, we need to show that the all the prototypical surfaces with  prototype in $\Pcal^A_D$ belong to the same $\GL^+(2,\R)$-orbit. For $D$ even (resp. $D\equiv 1 \mod 8$) and large enough, by Theorem~\ref{theo:H6:connect:PD}, we know that $\Pcal^A_D$ has two components, which means that we can not connect two prototypes in different components by using Butterfly moves.
Therefore, we need other moves to connect prototypes in $\Pcal^A_D$.
For that purpose, we will make use of prototypes in $\Pcal^B_D$.

Analogous to the Butterfly moves, we define the {\em Switch moves} $S_i, i \in \{1,2,3,4\}$,
from decompositions of type $B$ to decomposition of type $A$. They induce transformations
on the set of prototypes: $S_i:\Pcal^B_D \rightarrow \Pcal^A_D$.
The following proposition gives the admissibility conditions of the Switch moves.
\begin{Proposition}\label{prop:switch:move}
Let $(X,\omega)=X_D(w,h,0,e)$ be a surface with model $B$, that is $(w,h,0,e) \in \Pcal_D^\mathcal{B}$.
\begin{enumerate}
\item If $2h+e-w<0$ then the direction $\theta_1$ of slope $\frac{\lambda+h}{\lambda}$ on $(X,\omega)$
is a periodic direction of Model~A with prototype $S_1(w,h,0,e)=(w_1,h_1,t_1,e_1)$ satisfying
$$
e_1=3e-2w+4h.
$$
\item If $w-e-h < \lambda$ then the direction $\theta_2$  of slope $-\frac{\lambda+h}{\lambda}$ on $(X,\omega)$
is a periodic direction of Model~A with prototype $S_2(w,h,0,e)=(w_2,h_2,t_2,e_2)$ satisfying
$$
e_2=3e-2w+2h.
$$
\item If $3h+3e/2-w<0$ then the direction $\theta_3$  of slope $\cfrac{2\lambda+3h}{\lambda}$ on $(X,\omega)$
is a periodic direction of Model~A with prototype $S_3(w,h,0,e)=(w_3,h_3,t_3,e_3)$ satisfying
$$
e_3=7e+12h-4w.
$$
\item If $w-e-h < \lambda/2$ then the direction $\theta_4$ of slope $-2\frac{\lambda+h}{\lambda}$ on $(X,\omega)$
is a periodic direction of Model~A with prototype $S_4(w,h,0,e)=(w_4,h_4,t_4,e_4)$ satisfying
$$
e_4=5e-4w+4h.
$$
\end{enumerate}
\end{Proposition}
\begin{proof}[Proof of Proposition~\ref{prop:switch:move}]
We first assume $2h+e-w<0$. Clearly, the cylinder $\mathcal C_1$ in direction
$\theta_1$ as shown in Figure~\ref{fig:eigf:ModB:direc1} does exist if and only if
the quantity $y_1=(\lambda-w/2)\cdot \mathrm{slope}(\theta_1)$ satisfies $y_1<\lambda/2$
(and in this case $y_1$ is the height of $\mathcal C_1$).
A straightforward computation gives (recall that $wh=\lambda^2-e\lambda$):
$$
 y_1= (\lambda-w/2)\cdot \frac{\lambda+h}{\lambda} = \frac{2\lambda^2+(2h-w)\lambda-hw}{2\lambda} = \frac{2\lambda^2+(2h-w)\lambda-(\lambda^2-e\lambda)}{2\lambda}=
\frac{\lambda+(2h+e-w)}{2}.
$$
The assumption implies $y_1<\lambda/2$, thus there is a simple cylinder $\mathcal C_1$ and the direction $\theta_1$ is of Model~A.

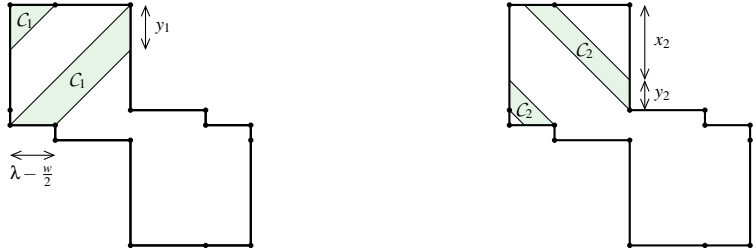
\begin{figure}[htb]
\begin{minipage}[l]{0.23\linewidth}
\centering
\begin{tikzpicture}[scale=0.2]
 \draw[thick] (0,16) -- (0,8) -- (3,8) --  (3,7) -- (8,7) -- (8,0) -- (16,0) -- (16,8) -- (13,8) -- (13,9) -- (8,9) -- (8,16) -- cycle;
\fill[yellow!10!green!10] (0,13) -- (0,16) -- (3,16) -- cycle;
\fill[yellow!10!green!10] (0,8) -- (8,16) -- (8,13) -- (3,8) -- cycle;

 \draw[thin, <->, >= angle 45] (9,16) -- (9,13);

 \draw (9,14.5) node[right] {\tiny $y_1$};

\draw[thick] (0,16) -- (0,8) -- (3,8) --  (3,7) -- (8,7) -- (8,0) -- (16,0) -- (16,8) -- (13,8) -- (13,9) -- (8,9) -- (8,16) -- cycle;

\draw[thin] (0,13) -- (3,16); (0,12) -- (4,16)  (0,11) -- (5,16) (0,10) -- (6,16)  (0,8) -- (8,16);
\draw[thin] (3,8) -- (8,13) ;
\draw[thin]     (0,8) -- (8,16);
 \foreach \x in {(0,16), (0,9), (0,8), (3,16), (3,8), (3,7), (8,16), (8,9), (8,7), (8,0), (13,9), (13,8), (13,0), (16,8), (16,7), (16,0)} \filldraw[fill=black] \x circle (4pt);

 \draw (4.5,11) node {\tiny $\mathcal C_1$} (1,15) node {\tiny $\mathcal C_1$};

 \draw[thin, <->, >= angle 45] (0,6) -- (3,6);

 \draw (1.5,6) node[below] {\tiny $\lambda-\frac{w}{2}$};
\end{tikzpicture}
\end{minipage}
\hskip 30mm
\begin{minipage}[l]{0.23\linewidth}
\centering
\begin{tikzpicture}[scale=0.2]
\fill[yellow!10!green!10] (1,16) -- (8,9) -- (8,11) -- (3,16) -- cycle;
\fill[yellow!10!green!10] (0,11) -- (0,9) -- (1,8) -- (3,8) -- cycle;

\draw[thick] (0,16) -- (0,8) -- (3,8) --  (3,7) -- (8,7) -- (8,0) -- (16,0) -- (16,8) -- (13,8) -- (13,9) -- (8,9) -- (8,16) -- cycle;

 \draw[thin]   (0,11) -- (3,8) (0,9) -- (1,8);
 \draw[thin] (1,16) -- (8,9) (3,16) -- (8,11);

 \foreach \x in {(0,16), (0,9), (0,8), (3,16), (3,8), (3,7), (8,16), (8,9), (8,7), (8,0), (13,9), (13,8), (13,0), (16,8), (16,7), (16,0)} \filldraw[fill=black] \x circle (4pt);

 \draw (5,13) node {\tiny $\mathcal C_2$} (1,9) node {\tiny $\mathcal C_2$};
 \draw[thin, <->, >= angle 45] (9,16) -- (9,11);
 \draw[thin, <->, >= angle 45] (9,11) -- (9,9);
 \draw (9,10) node[right] {\tiny $y_2$} (9,13.5) node[right] {\tiny $x_2$};

%
\end{tikzpicture}
\end{minipage}
\caption{Prototypical surface $(X,\omega)=S_D(w,1,0,e) \in \Omega E_D(6)$ of model B.
Cylinders in direction $\theta_1$ (left) and $\theta_2$ (right) are represented by $\mathcal C_1$ and $\mathcal C_2$ respectively.}
 \label{fig:eigf:ModB:direc1}
\end{figure}
Now by Lemma~\ref{lm:cyl} we have $\mathrm{Area(\mathcal C_1)}=(\lambda-\frac{w}{2})(\frac{\lambda}{2}+\frac{h}{2})=\frac{\lambda\cdot \lambda_1}{4}$. Since
$$
\lambda\cdot \lambda_1 = (2\lambda-w)(\lambda +h) = 2\lambda^2 + 2\lambda h -w\lambda - wh = \lambda^2 + (2h+e-w)\lambda
$$
we draw
$$
\lambda_1 = \lambda + 2h+e -w.
$$
Substituting $2\lambda=e+\sqrt{D}$ and $2\lambda_1=e_1+\sqrt{D}$ we obtain $e_1=4h+3e-2w$ as desired. \medskip

We now turn to the second assertion. As above we claim that the cylinder $\mathcal C_2$ in direction
$\theta_2$ exists if and only if the quantity $x_2=-\left(\frac{w-\lambda}{2}\right)\cdot \mathrm{slope}(\theta_2)$ satisfies $x_2<\lambda/2$ (and in
this case $y_2=\lambda/2-x_2$ is the height of $\mathcal C_2$).
Again a straightforward computation gives:
$$
x_2=\frac{(\lambda-w)}{2}\frac{(\lambda+h)}{\lambda}=\frac{\lambda^2 +\lambda h-w\lambda-\lambda^2+e\lambda}{2\lambda} = \frac{w-e-h}{2}.
$$
The assumption $w-e-h < \lambda$ implies $x_2<\lambda/2$ and there is a cylinder
$\mathcal C_2$ as desired. Since $\mathcal C_2$ is a simple cylinder, the direction $\theta_2$ is of Model A.
Now by Lemma~\ref{lm:cyl} we have $\mathrm{Area(\mathcal C_2)}=\frac{\lambda}{2}(\frac{\lambda}{2}-x_2)=\frac{\lambda\cdot \lambda_2}{4}$,
and
$$
\lambda_2 = \lambda -2x_2 = \lambda -(w-e-h).
$$
Since $2\lambda=e+\sqrt{D}$ and $2\lambda_2=e_2+\sqrt{D}$ we obtain $e_2=2h+3e-2w$. \medskip

For the third move we refer to Figure~\ref{fig:eigf:ModB:direc3}, left. The cylinder $\mathcal C_3$ exists if and only if $y_3 < \lambda/2$. On the other hand
a simple computation gives
$$
y_3=(\lambda-\frac{w}{2})\cdot\mathrm{slope}(\theta_3) = \frac{\lambda}{2}+3h+\frac{3e}{2}-w
$$
By the assumption, we have $y_3 < \frac{\lambda}{2}$, hence $\Cc$ exists. Now by Lemma~\ref{lm:cyl} we have $\mathrm{Area(\mathcal C_3)}=(\lambda-\frac{w}{2})\cdot (2\frac{\lambda}{2} + 3\frac{h}{2})=\frac{\lambda\cdot \lambda_3}{4}$.
Hence
$$
\lambda\cdot \lambda_3 = (2\lambda-w)\cdot (2\lambda+3h) = 2\lambda(2\lambda+3h) -2w\lambda - 3\lambda^2 + 3e\lambda.
$$
We draw
$$
\lambda_3 =  \lambda+6h -2w + 3e
$$
Substituting $2\lambda=e+\sqrt{D}$ and $2\lambda_3=e_3+\sqrt{D}$ we obtain $e_3=7e+12h-4w$ as desired.
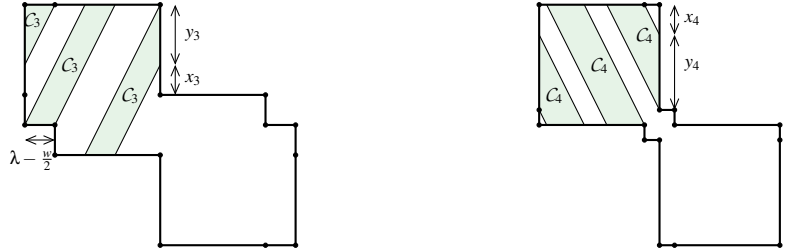
\begin{figure}[htb]
\begin{minipage}[l]{0.5\linewidth}
\centering
\begin{tikzpicture}[scale=0.2]
  \fill[yellow!10!green!10] (0,16) -- (0,12) -- (2,16) -- cycle;
  \fill[yellow!10!green!10] (0,8) -- (2,8) -- (6,16) -- (4,16) -- cycle;
  \fill[yellow!10!green!10] (4,6) -- (6,6) -- (9,12) -- (9,16) -- cycle;

 \draw[thick] (0,16) -- (0,8) -- (2,8) --  (2,6) -- (9,6) -- (9,0) -- (18,0) -- (18,8) -- (16,8) -- (16,10) -- (9,10) -- (9,16) -- cycle;
 \draw[thin] (0,12) -- (2,16) (0,8) -- (4,16) (2,8) -- (6,16)  (4,6) -- (9,16) (6,6) -- (9,12) ;
 \foreach \x in {(0,16), (0,10), (0,8), (2,16), (2,8), (2,6), (9,16), (9,10), (9,6), (9,0), (16,10), (16,8), (16,0), (18,8), (18,6), (18,0)} \filldraw[fill=black] \x circle (4pt);
  \foreach \x in {(0.5,15), (3,12), (7,10)} \draw \x node {\tiny $\mathcal C_3$};
  \draw[thin, <->, >= angle 45] (10,16) -- (10,12);
  \draw[thin, <->, >= angle 45] (10,12) -- (10,10);absence of such cylinders
  \draw (10,14) node[right] {\tiny $y_3$} (10,11) node[right] {\tiny $x_3$};
 \draw[thin, <->, >= angle 45] (0,7) -- (2,7); \draw (0.5,7) node[below] {\tiny $\lambda-\frac{w}{2}$};
\end{tikzpicture}
\end{minipage}
\hskip 10mm
\begin{minipage}[l]{0.23\linewidth}
\centering
\begin{tikzpicture}[scale=0.2]

 \fill[yellow!10!green!10] (0,14) -- (0,9) -- (0.5,8) -- (3,8) -- cycle;
 \fill[yellow!10!green!10] (0.5,16) -- (4.5,8) -- (7,8) -- (3,16) -- cycle;
 \fill[yellow!10!green!10] (4.5,16) -- (8,9) -- (8,14) -- (7,16) -- cycle;
 \draw[thick] (0,16) -- (0,8) -- (7,8) --  (7,7) -- (8,7) -- (8,0) -- (16,0) -- (16,8) -- (9,8) -- (9,9) -- (8,9) -- (8,16) -- cycle;
 \draw[thin]  (0,14) -- (3,8) (0,9) -- (0.5,8);
 \draw[thin] (0.5,16) -- (4.5,8) (3,16) -- (7,8) (4.5,16) -- (8,9) (7,16) -- (8,14) ;
 \foreach \x in {(0,16), (0,9), (0,8), (7,16), (7,8), (7,7), (8,16), (8,9), (8,7), (8,0), (9,9), (9,8), (9,0), (16,8), (16,7), (16,0)} \filldraw[fill=black] \x circle (4pt);
 \foreach \x in {(1,10), (4,12), (7,14)} \draw \x node {\tiny $\mathcal C_4$};
 \draw[thin, <->, >= angle 45] (9,16) -- (9,14);
 \draw[thin, <->, >= angle 45] (9,14) -- (9,9);
 \draw (9,12) node[right] {\tiny $y_4$} (9,15) node[right] {\tiny $x_4$};
\end{tikzpicture}
\end{minipage}
\caption{Cylinders in direction $\theta_3,\theta_4$: cylinders $\mathcal C_3,\mathcal C_4$ correspond to the shaded regions.}
\label{fig:eigf:ModB:direc3}
\end{figure}

We now turn to the last assertion. Applying the same remark as above,
the cylinder $\mathcal C_4$ as shown in Figure~\ref{fig:eigf:ModB:direc3} exists if and only if $x_4 < \lambda/2$.
On the other hand a simple computation gives $x_4=-\frac{w-\lambda}{2}\cdot\mathrm{slope}(\theta_4) = w-e-h$.
Thus by the assumption, $\Cc_4$ exists and the direction of $\Cc_4$ is of Model~A.

Now by Lemma~\ref{lm:cyl} we have $\mathrm{Area(\mathcal C_4)}=(\frac{\lambda}{2}-x_4)\cdot \frac{\lambda}{2}=\frac{\lambda\cdot \lambda_4}{4}$.
Hence
$$
\lambda_4 = 2 \cdot (\frac{\lambda}{2}-x_4).
$$
Substituting $2\lambda=e+\sqrt{D}$ and $2\lambda_4=e_4+\sqrt{D}$, we obtain $e_4=e-4x_4=5e-4w+4h$ as desired.
\end{proof}


\section{Proof of Theorem~\ref{thm:H6:eig:comp} for $D$ even and not a square}\label{sec:pf:H6:conn:D:even:n:sq}
In this section, we will show
\begin{Theorem}\label{thm:D:even:n:sq}
For any even discriminant $D\geq 8$ that is not a square, $\Omega E_D(6)$ is connected.
\end{Theorem}

By Theorems~\ref{theo:onto:map} and~\ref{theo:H6:connect:PD}, it is enough to find a surface  $(X,\omega) \in \Omega E_D(6)$,
on which there exist two periodic directions such that the corresponding cylinder decompositions are both in Model A,  and the associated prototypes $p_i=(w_i,h_i,t_i,e_i), \ i=1,2$, satisfy $e_1-e_2\equiv 2 \mod 4$.

Our strategy is to look for a prototypical surface $(X,\omega)=X_D(w,h,t,e)\in \Omega E_D(6)$ having two simple cylinders $\mathcal C_1,\mathcal C_2$ in two different directions, say $\theta_1$ and $\theta_2$, for which one has
$$
\frac{8}{\lambda}\left(\mathrm{Area}(\mathcal C_1) - \mathrm{Area}(\mathcal C_2)\right) \not \equiv 0 \mod 4, \text{ where } \lambda=\frac{e+\sqrt{D}}{2}.
$$
Indeed, the corresponding cylinder decompositions associated to $\theta_1,\theta_2$ are of Model A with prototypes
$(w_1,h_1,t_1,e_1)$ and $(w_2,h_2,t_2,e_2)$. By Lemma~\ref{lm:cyl} one has
$\mathrm{Area}(\mathcal C_1) - \mathrm{Area}(\mathcal C_2)=\lambda/8(e_1-e_2)$.
Theorem~\ref{theo:H6:connect:PD} then implies
that all the prototypical surfaces of Model A belong to the same $\GL^+(2,\R)$-orbit. Since any $\GL^+(2,\R)$-orbit contains
a prototypical surface of Model A (by Proposition~\ref{prop:H6:noModA}), this will prove the theorem.

To this end we will use Proposition~\ref{prop:switch:move}. We will find $(w,h,t,e)\in \Pcal^B_D$ such that
there are $i,j \in \{1,2,3,4\}$ for which $e_i-e_j\equiv 2 \mod 4$ where $S_i(w,h,t,e)=(w_i,h_i,t_i,e_i)$ and $S_j(w,h,t,e)=(w_j,h_j,t_j,e_j)$.

\begin{proof}[Proof of Theorem~\ref{thm:D:even:n:sq}]
For $D\in \{8,12\}$, the theorem follows from Theorem~\ref{theo:onto:map} and Theorem~\ref{theo:H6:connect:PD}.
From now on we assume that $D\geq 20$ is a non square even discriminant.

We first assume that $D$ is not an exceptional
discriminant in Theorem~\ref{theo:H6:connect:PD}, namely $D\not \in \{52,68,84\}$.
Since $D$ is not a square, there is a unique natural number $e$ such that $e+2 < \sqrt{D} < e+4$ and $D \equiv e \mod 2$.
Then $(w,h,t,e)=(\frac{D-e^2}{4},1,0,e) \in \Pcal_D$. The condition $e+2 < \sqrt{D} < e+4$ is
equivalent to $ w/2 < \lambda < w$ thus $(w,h,t,e) \in \Pcal^B_D$. Let $(X,\omega):=X_D(w,1,0,e)$. \medskip

In view of applying Proposition~\ref{prop:switch:move} we rewrite the admissibility conditions of $S_1,S_2$ in terms of $D$:
$$
\left\{
\begin{array}{lll}
(e+2)^2 + 4 < D & \iff & 2h+e-w<0 \\
(e+2)^2 < D< (e+4)^2-4 & \implies & w-e-h < \lambda
\end{array}
\right.
$$
%
%
%
Since $D$ is an even discriminant satisfying $(e+2)^2<D<(e+4)^2$, one of the following holds: \medskip

\noindent {\bf First case: $(e+2)^2+4 < D < (e+4)^2-4$}.\\
$S_1$ and $S_2$ are admissible and we have: $e_1=3e -2w+4h$ and $e_2=3e-2w+2h$. Since $h=1$, we have that $e_1 -e_2 \equiv 2 \mod 4$. \medskip

\noindent {\bf Second case: $D=(e+4)^2-4$}.\\
$S_1$ is admissible and $e_1=3e-2w+4$. Since $w=\frac{D-e^2}{4} = 2e+3$ we draw $e_1=-e-2$.\\
Now $3h+3e/2-w=-e/2<0$. Hence $S_3$ is also admissible. We obtain $e_3 = 7e+12h-4w \equiv -e \mod 4$.
Again this gives $e_1-e_3 \equiv 2  \mod 4$.\medskip

\noindent {\bf Third case: $D=(e+2)^2+4$}.\\
Since $(e+2)^2 < D < (e+4)^2-4$, the move $S_2$ is admissible, and $e_2=3e-2w+2$. Since $w=\frac{D-e^2}{4} = e+2$ we draw $e_2=e-2$.\\
Now,  $w-e-h=1<\lambda/2$, hence the move $S_4$ is also admissible
and $e_4=5e-4w+4h=e-4$. We conclude $e_2-e_4=2$. \medskip

It remains to prove the theorem for the three exceptional cases $D \in \{52,68,84\}$. This is discussed in detail in Appendix~\ref{sec:except:1}. The proof of Theorem~\ref{thm:D:even:n:sq} is now complete.
\end{proof}


\section{Proof of Theorem~\ref{thm:H6:eig:comp} for $D=d^2$, with $d$ even}\label{sec:main:thm:D:square:even}

We now provide a proof of Theorem~\ref{thm:H6:eig:comp}  when $D$ is a square and even.
\begin{Theorem}
\label{thm:D:square:even}
For any even discriminant $D=d^2$ where $d\geq 14$, the locus $\Omega E_{D}(6)$ is connected.
\end{Theorem}
\begin{proof}[Proof of Theorem~\ref{thm:D:square:even}]
We will construct a  surface $(X,\omega)$ as shown in Figure~\ref{fig:prym6:2cyl:dec}.
Observe that $X$ admits an involution $\tau$  that exchanges the two horizontal cylinders such that $\tau^*\omega=-\omega$.
Since $\tau$ has two fixed points, one of which is the unique zero of $\omega$, $(X,\omega)$ is a Prym from in $\mathcal{H}(6)$.

For $\alpha \in \{A,\bar{A},B,\bar{B},C,\bar{C}\}$, let $ l_\alpha$ denote the length of $\alpha$.
Note that, for $\alpha \in \{A,B,C\}$, $\tau$ exchanges $\alpha$ and $\bar{\alpha}$, therefore $l_\alpha=l_{\bar{\alpha}}$.
The heights of the two horizontal cylinders are set to be $1$.

\begin{figure}[!htb]
\centering
\begin{tikzpicture}[scale=0.28]
\fill[yellow!10!green!20] (8,0) --(10,0) -- (15,5) -- (15,7) -- cycle;
\fill[yellow!10!green!20] (0,7) -- (0,5) -- (3,8) -- (1,8) -- cycle;
\fill[yellow!10!green!20] (1,4) -- (3,4) -- (7,8) -- (5,8) -- cycle;
\draw[pattern=dots] (5,8) -- (8,0) -- (10,0) -- (7,8) -- cycle;

\draw (0,8) --(0,4) -- (5,4) -- (5,0) -- (20,0) -- (20,4) -- (15,4) -- (15,8) -- cycle;

\draw (0,7) -- (1,8) (0,5) -- (3,8) (1,4) -- (5,8) (3,4) -- (7,8) (8,0) -- (15,7) (10,0) -- (15,5);
\foreach \x in {(0,8), (0,4), (5,8), (5,4), (5,0), (7,8), (8,0), (10,8), (10,0), (12,8), (13,0), (15,8), (15,4), (15,0), (20,4), (20,0)} \filldraw[fill=black] \x circle (3pt);
\draw (2.5,8.5) node {\tiny $A$} (2.5,3.5) node {\tiny $A$};
\draw (6,8.5) node {\tiny $B$} (9,-0.5) node {\tiny $B$};
\draw (8.5, 8.5) node {\tiny $C$} (6.5,-0.5) node {\tiny $C$};
\draw (17.5,4.5) node {\tiny $\bar{A}$} (17.5,-0.5) node {\tiny $\bar{A}$};
\draw (11,8.5) node {\tiny $\bar{B}$} (14,-0.5) node {\tiny $\bar{B}$};
\draw (13.5,8.5) node {\tiny $\bar{C}$} (11.5,-0.5) node {\tiny $\bar{C}$};

\draw (7.5,4) node {\tiny $\Cc'$};
\draw (12,3) node {\tiny $\Cc$};
\draw (1,7) node {\tiny $\Cc$};
\draw (4,6) node {\tiny $\Cc$};
\end{tikzpicture}

\caption{A surface $(X,\omega) \in  \mathrm{Prym}(6)$ with two simple cylinders
$\mathcal C$ and $\mathcal C'$.}
\label{fig:prym6:2cyl:dec}
\end{figure}
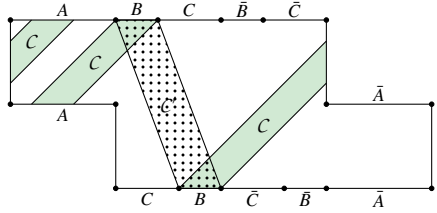
Elementary computation shows that the slope of the cylinder $\mathcal C$ is $\frac{3}{l_A+2l_B+l_C}$, and $\Cc$  exists if and only if the following inequalities hold:
$$
2l_B+l_C < 2l_A \qquad \textrm{ and } \qquad l_A < l_B + 2l_C
$$
or equivalently
\begin{equation}
\label{eq:cylinder}
l_A -2l_C < l_B < l_A-\frac1{2}l_C.
\end{equation}
Let us fixed a natural number $d$. For a given $l_B\in \N$, we let $l_C=l_B-1$ and $l_A = d-2l_B-2l_C=d-4l_B+2$.  Equation~\eqref{eq:cylinder} is then equivalent to
\begin{equation}
\label{eq:cylinder:2}
\frac1{7}d + \frac{4}{7} < l_B < \frac{2}{11}d + \frac{5}{11}
\end{equation}
Observe that if  there exists $l_B \in \N$ such that \eqref{eq:cylinder:2} holds then $l_A,l_B,l_C>0$. \medskip

If  $d$ is sufficiently large, for instance $d > 55 \Rightarrow \frac{2}{11}d + \frac{5}{11} - (\frac1{7}d + \frac{4}{7}) > 2$,
then there exists $l_B\in \N$, $l_B$ odd, such that \eqref{eq:cylinder:2} holds.
For $14 \leq d \leq 54$, we check that there exists $l_B$ odd such that \eqref{eq:cylinder:2} holds
if $d \not \in \{14,18,20,22,24,32,34,36,46\}$.

We first assume that $d \geq 14$ and  $d \not \in \{14,18,20,22,24,32,34,36,46\}$.
Then there exist $l_A,l_B,l_C \in \N$ such that
$$
\left\{
\begin{array}{l}
l_A-2l_C< l_B < l_A-\frac{1}{2}l_C,\\
l_C=l_B-1,\\
l_A+2l_B+2l_C=d,\\
l_B \textrm{ is odd}.
\end{array}
\right.
$$
Let $(X,\omega)$ be the surface constructed from the parameters $l_A,l_B,l_C$ as above,
and $h=1$, where $h$ is the height of both horizontal cylinders.
Since $(X,\omega)$ is square-tiled, its Veech group contains hyperbolic elements. Thus $(X,\omega)$ is a Prym eigenform
in $\Omega E_D(6)$, with $D$ being a square (see \cite{Mc7}).
Since $\gcd(l_B,l_C)=1$ and $h=1$, $(X,\omega)$ is primitive. A direct computation gives
$\mathrm{Area(X,\omega)}=2d$. Thus $(X,\omega) \in \Omega E_{d^2}(6)$ by Proposition~\ref{prop:square:tiled}.

Now the cylinder $\mathcal C$ is simple so that by Corollary~\ref{cor:cylinder:lambda}
there is $g \in \GL^+(2,\R)$ such that $g\cdot (X,\omega)=X_{d^2}(w,h,t,e_B)$, with $(w,h,t,e_B) \in \Pcal^A_{d^2}$, and
$$
\mathrm{Area}(\mathcal C) = 3l_B = \frac{d+e_B}{2}
$$
On the other hand, the cylinder $\mathcal C'$ is also simple, thus there is $g'$ such that $g'\cdot(X,\omega)=X_{d^2}(w',h',t',e'_B)$ and
$$
\mathrm{Area}(\mathcal C') = 2l_B = \frac{d+e'_B}{2}
$$
We draw
$$
e_B-e'_B = 2l_B \equiv 2 \mod 4
$$
since $l_B$ is odd. Thus the two components of the set of prototypes in $\Pcal^A_{d^2}$ are connected. This proves Theorem~\ref{thm:D:square:even}
for $d \not\in \{14,18,20,22,24,32,34,36,46\}$ . \medskip

A short argument handles the remaining cases by using specific prototype of model $B$ satisfying Proposition~\ref{prop:switch:move}
as follows: observe that for a prototype $(w,h,0,e) \in \Pcal^B_{d^2}$, the moves $S_1$ and $S_2$ are admissible if and only if
$$
h < w-e-h<\frac{e+d}{2}
$$
For each exceptional $d$, we find a find a suitable $(w,h,0,e)\in \Pcal^B_{d^2}$ where $h$ is odd. This will give
$S_1(w,h,0,e)=(w_1,h_1,t_1,e_1)$ and $S_2(w,h,0,e)=(w_2,h_2,t_2,e_2)$ with
$$
e_1 -e_2 =2h\equiv 2 \mod 4
$$
concluding the proof of the theorem. This is done in Table~\ref{table:except:Dsq:even:link} below.
\begin{table}[htbp]
$$
\begin{array}{ll}
\begin{array}{|c|c|c|}
\hline
d & \textrm{$(w,h,t,e) \in \mathcal \Pcal^B_{D}$} & h<w-e-h < (e+d)/2 \\
\hline
\hline
14& {\scriptstyle (15, 3, 0, 4)}& {\scriptstyle 3<8<9} \\
18& {\scriptstyle (16, 5, 0, 2)}& {\scriptstyle 5<9<10} \\
20& {\scriptstyle (25, 3, 0, 10)}& {\scriptstyle 3<12<15} \\
22& {\scriptstyle (21, 5, 0, 8)}& {\scriptstyle 5<8<15} \\
24& {\scriptstyle (20, 7, 0, 4)}& {\scriptstyle 7<9<14} \\
\hline
\end{array} &
\begin{array}{|c|c|c|}
\hline
d & \textrm{$(w,h,t,e) \in \mathcal \Pcal^B_{D}$} & h<w-e-h < (e+d)/2 \\
\hline
\hline
32& {\scriptstyle (28, 9, 0, 4)}& {\scriptstyle 9<15<18} \\
34& {\scriptstyle (45, 5, 0, 16)}& {\scriptstyle 5<24<25} \\
36& {\scriptstyle (35, 9, 0, 6)}& {\scriptstyle 9<20<21} \\
46& {\scriptstyle (35, 15, 0, 4)}& {\scriptstyle 15<16<25} \\
&&\\
\hline
\end{array}
\end{array}
$$
\caption{
Connecting  components of $\Pcal^A_D$ through model $B$ for exceptional discriminants $D=d^2$.
}
\label{table:except:Dsq:even:link}
\end{table}
\end{proof}


\section{Proof of Theorem~\ref{thm:H6:eig:comp} when $D\equiv 1 \mod 8$}
\label{sec:D1mod8}

In this section we prove Theorem~\ref{thm:H6:eig:comp} for $D \equiv 1 \mod 8$.
\begin{Theorem}
\label{thm:D1mod8:connect:PA1:PA2}
For any discriminant $D \equiv 1 \mod 8$, $D>9$, $\Omega E_D(6)$ contains a single $\GL^+(2,\R)$-orbit.
\end{Theorem}

\subsection{Connecting $\Pcal^{A_1}_D$ and $\Pcal^{A_2}_D$ for generic values of $D$}\label{sec:D1mod8:generic}
Let us introduce some necessary material for the proof.
For $D\equiv 1\mod 8$  large enough, we know by Theorem~\ref{theo:H6:connect:PD} that $\Pcal^A_D$ has two components $\Pcal^{A_1}_D$ and $\Pcal^{A_2}_D$.
Let $(X,\omega):=X_D(w,h,0,e)$ be a prototypical surface, where $(w,h,0,e) \in \Pcal_D^{A_2}$.
To prove Theorem~\ref{thm:D1mod8:connect:PA1:PA2}, it is sufficient to find a periodic direction $\theta$ with prototype
$(w',h',t',e') \in \Pcal_D^{A_1}$.
However such a direction is rather difficult to exhibit. We will work on the universal cover of $(X,\omega)$
to find a simple cylinder  with associated prototype in $\Pcal^{A_1}_D$.
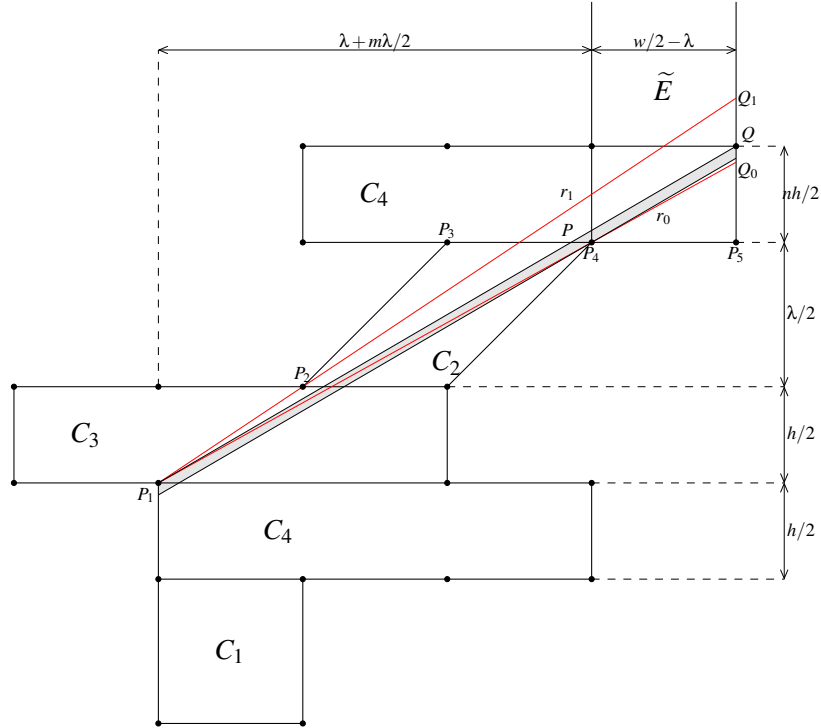
\begin{figure}[htb]
\begin{tikzpicture}[scale=0.32]
\fill[blue!50!yellow!20] (6,4) -- (intersection of 0,0--24,14 and 6,0--6,4) -- (30,17.5) -- (30,18) -- cycle;
\draw[thin] (24,24) -- (24,18) -- (12,18) -- (12,14) -- (18,14) -- (12,8) -- (0,8) -- (0,4) -- (6,4) -- (6,-6) --(12,-6) -- (12,0) -- (24,0) -- (24,4) -- (18,4) -- (18,8) -- (24,14) -- (30,14) -- (30,24);
\draw[thin] (24,18) -- (30,18) ;
\draw[thin] (18,14) -- (24,14) -- (24,18);
\draw[thin] (12,8) -- (18,8);
\draw[thin] (6,4) -- (18,4);
\draw[thin] (6,0) -- (12,0);
\draw[very thin, red] (6,4) -- (30, 20);
\draw[very thin, red ] (6,4) -- (intersection of 24,14-- 33,19 and 30,14--30,18);
\draw[very thin] (6,4) -- (30,18);
\draw[very thin] (intersection of 0,0--24,14 and 6,0--6,4) -- (30,17.5);
\foreach \x in {(0,8),(0,4), (6,8),(6,4),(6,0),(6,-6),(12,18), (12,14),(12,8),(12,0),(12,-6),(18,18), (18,14), (18,8), (18,4), (18,0), (24,18), (24,14), (24,4), (24,0), (30,18), (30,14)} \filldraw[fill=black] \x circle (3pt);

\draw[very thin, <->, >= angle 45] (6,22) -- (24,22);
\draw[very thin, dashed]  (6,8) -- (6,22);

\draw[very thin, dashed] (30,18) -- (32,18);
\draw[very thin, dashed] (30,14) -- (32,14);
\draw[very thin, dashed] (18,8) --(32,8);
\draw[very thin, dashed] (24,4) -- (32,4);
\draw[very thin, dashed] (24,0) -- (32,0);
\draw[very thin, <->, >= angle 45] (24,22) -- (30,22);
\draw[very thin, <->, >= angle 45] (32,18) -- (32,14);
\draw[very thin, <->, >= angle 45] (32,14) -- (32,8);
\draw[very thin, <->, >= angle 45] (32,8) -- (32,4);
\draw[very thin, <->, >= angle 45]  (32,4) -- (32,0);

\draw (15,22.3) node {\tiny $\lambda+m\lambda/2$}  (27,22.3) node {\tiny $w/2-\lambda$};
\draw (32.7, 16) node {\tiny $nh/2$} (32.7,11) node {\tiny $\lambda/2$} (32.7,6) node {\tiny $h/2$} (32.7,2) node {\tiny $h/2$};

\draw (27,20.5) node {$\widetilde{E}$} (15,16) node {$C_4$} (18,9) node {$C_2$} (3,6) node {$C_3$} (11,2) node {$C_4$}  (9,-3) node {$C_1$};
\draw (5.5,3.5) node {\tiny $P_1$} (12,8.5) node {\tiny $P_2$} (18,14.5) node {\tiny $P_3$} (24,13.5) node {\tiny $P_4$} (30,13.5) node {\tiny $P_5$};
\draw (23,14.5) node {\tiny $P$} (30.5,20) node {\tiny $Q_1$} (30.5,17) node {\tiny $Q_0$} (30.5,18.5) node {\tiny $Q$};
\draw (23,16) node {\tiny $r_1$} (27,15) node {\tiny $r_0$};
\end{tikzpicture}
\caption{Searching for periodic directions of model $A$ with prototype in $\Pcal_D^{A_1}$: the shaded region corresponds to a simple cylinder.}
 \label{fig:D1mod8:connect:PA1:PA2}
\end{figure}

In what follows, we will refer to Figure~\ref{fig:D1mod8:connect:PA1:PA2}.
We denote the ray starting from $P_1$ and passing through $P_2$ by $r_1$. Its direction is $\theta_1$
and its slope is
$$
k_1=\frac{h}{\lambda}.
$$
This ray eventually exits the cylinder $C_2$ through its top border.

\begin{Lemma}\label{lm:D1mod8:new:dir:m}
On the universal cover, there is a horizontal segment  $\ol{P_3P_4}$ representing
the top border of $C_2$ ($P_3,P_4$ correspond to the unique singularity of $X$) that intersects $r_1$.
As a vector in $\R^2$, we have $\overrightarrow{P_1P_4}=(\lambda+m\frac{\lambda}{2},\frac{h}{2}+\frac{\lambda}{2})$,
with $m=\lfloor\frac{\lambda}{h}\rfloor \in \N \cup\{0\}$, where $\lfloor.\rfloor$ is the integral part function.
Note that $m$  is the number of times $r_1$ intersects the unique vertical saddle connection in $C_2$.
\end{Lemma}
\begin{proof}
We have $\overrightarrow{P_1P_4}=(\lambda+m\frac{\lambda}{2},\frac{h}{2}+\frac{\lambda}{2})$,
where $m\in \N \cup\{0\}$  is the number of times $r_1$ intersects the unique vertical saddle connection in $C_2$.

Let $P$ be the intersection of $r_1$ and $\ol{P_3P_4}$.
Comparing the horizontal components of the vectors $\overrightarrow{P_1P_3}, \overrightarrow{P_1P}, \overrightarrow{P_1P_4}$, we have
 $$
 (m+1)\frac{\lambda}{2} \leq  (\frac{h}{2}+\frac{\lambda}{2})\cdot\frac{\lambda}{h} < (m+2)\frac{\lambda}{2} \Leftrightarrow m \leq \frac{\lambda}{h} < m+1.
 $$
\end{proof}

The ray from $P_1$ which passes through $P_4$ is denoted $r_0$. Its direction is $\theta_0$
and its slope is
$$
k_0=\frac{h/2+\lambda/2}{\lambda+m \lambda/2}=\frac{h+\lambda}{(m+2)\lambda}.
$$
Since the top of $C_2$ is glued to the bottom of $C_4$, we draw a copy of $C_4$ above $C_2$.
The ray $r_1$ then enters $C_4$ and crosses the left border of the vertical simple cylinder $E$ that is contained in $\ol{C_4}$.
We can represent the universal cover $\tilde{E}$ of $E$ as an infinite vertical band intersecting this copy of $C_4$
in a rectangle representing $E$. Let $Q_1$ denote the intersection of $r_1$ with the right border of $\tilde{E}$.
The ray $r_0$ also crosses $\tilde{E}$. We denote its intersection with the right border of $\tilde{E}$ by $Q_0$.

For $i=1,2$, we define  the {\em  $x$-coordinate ({\rm resp.} $y$-coordinate) of $Q_i$} to be the horizontal (resp. vertical) component of the vector $\overrightarrow{P_1Q_i}$.
In other words, these are the coordinates of $Q_i$ in the plane with origin being $P_1$.
An easy computation shows that the $y$ coordinate of $Q_i$  is
$$
(Q_i)_y = k_i \cdot \frac{w+m\lambda}{2}.
$$
The next lemma gives a sufficient condition to ensure the existence of a simple cylinder.
\begin{Lemma}
\label{lm:existence}
If there exists $n\in \N$ such that
$$
(Q_0)_y < \cfrac{\lambda}{2} + (n+1)\frac{h}{2} < (Q_1)_y
$$
then there is a simple cylinder in direction with slope $k=\frac{(n+1)h+\lambda}{w+m\lambda}$.
\end{Lemma}
\begin{proof}
The assumption means that the segment $\ol{Q_0Q_1}$ contains a pre-image $Q$ of the singularity of $X$.
Note that the distance from $Q$ to the bottom right vertex of the rectangle representing $C_4$ equals $nh/2$.
Let $\theta$ be the direction of $\ol{P_1Q}$. Then the slope of $\theta$ is $k=\frac{(n+1)h+\lambda}{w+m\lambda}$.
One can easily check that the segment $\ol{P_1Q}$ represents a saddle connection in $M$
which is a boundary component of a simple cylinder $C$.
The other boundary component of $C$ is represented by a segment in direction $\theta$ passing through $P_4$.
\end{proof}

\begin{Lemma}\label{lm:D1mod8:new:dir:par}
Let $(w',h',t',e')$
be the prototype associated to the cylinder decomposition in direction $\theta$.
Then
$$
e'   = 3e-2w+4h+2n(m+2)h.
$$
Moreover if $w \equiv n\cdot m \cdot h \mod 4$ then $(w',h',t',e') \in \Pcal_D^{A_1}$.
\end{Lemma}
\begin{proof}
Let $P$ be the intersection of $\ol{P_1Q}$ and $\ol{P_3P_4}$. We first compute $x=|\ol{PP_4}|$.
We have
 $$
 \begin{array}{ccl}
  x & = & (m+2)\frac{\lambda}{2}-(\frac{h}{2}+\frac{\lambda}{2})\cdot\frac{1}{k}\\
    & = & \dfrac{\lambda}{2}\left((m+2)- \dfrac{(m+1)\lambda-e  + w + mh}{(n+1)h+\lambda}\right) \mbox{(here, we used the fact that $\lambda^2=e\lambda+wh$)}\\
    & = & \dfrac{\lambda}{2} \cdot \dfrac{\lambda+e-w+2h+n(m+2)h}{(n+1)h+\lambda}
 \end{array}
$$
Now, the area of $C$ is
$$
\Aa(C)=x\cdot\frac{(n+1)h+\lambda}{2}=\frac{\lambda}{4}\cdot(\lambda+e-w+2h+n(m+2)h).
$$
The first assertion then follows from Lemma~\ref{lm:cyl}.

We now prove the second assertion of the lemma.
Recall that $(w,h,0,e)\in \Pcal^{A_2}_D$, which means that $w$ and $h$ are even.
Therefore, $4\mid wh=\frac{D-e^2}{4}$.
Since $D \equiv 1 \mod 8$, we have two cases: if $D\equiv  1 \mod 16$ then $e \equiv \pm 1 \mod 8$, and if $D \equiv 9 \mod 16$ then $e \equiv \pm 3 \mod 8$.
The assumption then implies that $e' \equiv 3e \mod 8$.  An elementary computation shows that in either case $4\nmid \frac{D-{e'}^2}{4}$,
 which means that $w'$ and $h'$ cannot be both even, hence $(w',h',t',e') \in \Pcal_D^{A_1}$.
\end{proof}

\begin{Proposition}
\label{prop:D1mod8:connect:Dlarge}
For any $D \equiv 1 \mod 8$ with $D>9$ and $D\not\in \{17,25,33,41,49,65,73,105\}$ there exists $(w,2,0,e) \in \Sc^2_D$ such that there is a simple cylinder in direction $\theta$ with associated prototype in $\Pcal^{A_1}_D$.
\end{Proposition}
For the proof of Proposition~\ref{prop:D1mod8:connect:Dlarge}, we first need the following
\begin{Lemma}\label{lm:D1mod8:good:par}
Assume that  $D > 21^2$, then there exists a prototype
$(w,2,0,e) \in \Sc^2_D$ with $e \in (-\sqrt{D},-\sqrt{D}+21)$ such that
$$
\left\{
\begin{array}{l}
(\lfloor\frac{\lambda}{2}\rfloor+1)- \frac{\lambda}{2} \geq \frac{1}{2}, \text{ and }\\
 4 \mid w.
\end{array}
\right.
$$
\end{Lemma}
\begin{proof}
Since $D \equiv 1 \mod 8$, we have $D \equiv 1,9 \mod 16$.
For the rest of this prove we will assume that $D\equiv 1 \mod 16$,
the other case follows from the same argument.

\noindent \underline{Step 1:} let $e_0 \in (-\sqrt{D},-\sqrt{D}+7)$ be an integer such that $e_0\equiv \pm 1 \mod 8$.
 If $\frac{e_0+\sqrt{D}}{4}-\lfloor\frac{e_0+\sqrt{D}}{4}\rfloor \leq \frac{1}{2}$ then we choose $e_1=e_0$.
Otherwise, either $e_1=e+2$ or $e_1=e+6$ satisfies $e_1\equiv \pm 1 \mod 8$. Note that in either case, we have
$$
\frac{e_1+\sqrt{D}}{4}-\lfloor\frac{e_1+\sqrt{D}}{4}\rfloor= \frac{e_0+\sqrt{D}}{4}-\lfloor\frac{e_0+\sqrt{D}}{4}\rfloor-\frac{1}{2} < \frac{1}{2}.
$$
Thus there exists $e_1 \in (-\sqrt{D},-\sqrt{D}+13)$ such that
$$
(\lfloor\frac{e_1+\sqrt{D}}{4}\rfloor+1)-\frac{e_1+\sqrt{D}}{4} \geq \frac{1}{2}.
$$

\noindent \underline{Step 2:} consider now $w_1=\frac{D-e^2_1}{8}$. Note that by assumption, $w_1$ is even.
If $4 \mid w_1$, then $(w_1,2,0,e_1)$ is the desired prototype. Otherwise, consider $e_2=e_1+8 \in (-\sqrt{D},-\sqrt{D}+21)$. We have
$$
w_2:=\frac{D-e_2^2}{8}=\frac{D-(e_1+8)^2}{8}=w_1+2e_1+8.
$$
Since $e_1$ is odd, we have $4 | w_2$.
Moreover, we also have
$$
\frac{e_2+\sqrt{D}}{4}-\lfloor\frac{e_2+\sqrt{D}}{4}\rfloor=\frac{e_1+\sqrt{D}}{4}-\lfloor\frac{e_1+\sqrt{D}}{4}\rfloor \leq \frac{1}{2}
$$
Therefore, $(w_2,2,0,e_2)$ is the desired prototype.
\end{proof}

\begin{proof}[Proof of Proposition~\ref{prop:D1mod8:connect:Dlarge}]
Lemma~\ref{lm:D1mod8:new:dir:par} provides us with a sufficient condition to guarantee that the prototype of the direction $\theta$ belongs to $\Pcal^{A_1}_D$. For any fixed $D\equiv 1 \mod 8$, we can use this criterion to check Proposition~\ref{prop:D1mod8:connect:Dlarge} for $9 < D \leq 57^2=3249$. Thus assume that $D>57^2$.

 Let $(w,2,0,e)\in\Pcal_D^{A_2}$, with $e \in(-\sqrt{D},-\sqrt{D}+21)$ that satisfies the conditions of Lemma~\ref{lm:D1mod8:good:par}. If there is $n\in \N$ such that
 $$
(Q_0)_y < \cfrac{\lambda}{2} + (n+1)\frac{h}{2} < (Q_0)_y
$$
then Lemma~\ref{lm:existence} implies the existence of a simple cylinder and
a prototype $(w',h',t',e') \in\Pcal_D^{A}$. By Lemma~\ref{lm:D1mod8:new:dir:par}
$(w',h',t',e') \in\Pcal_D^{A_1}$ if $w \equiv n\cdot m \cdot h \mod 4$. Since $w\equiv 0 \mod 4$
it suffices to show that $n$ can be chosen even. This is obviously the case if we have
$$
y:=(Q_1)_y-(Q_0)_y > 2\cdot\frac{h}{2}=2.
$$
By construction, the left hand side of the above inequality is (recall $h=2$):
$$
\begin{array}{ccl}
y & = & \dfrac{w+m\lambda}{2}\times (k_1-k_0), \\
  & = & \dfrac{w+m\lambda}{2(m+2)}\times\dfrac{h}{\lambda}\times\left((m+1)-\dfrac{\lambda}{h}\right)\\
  & = & \dfrac{\lambda-e+mh}{2(m+2)}\times\left((m+1)-\dfrac{\lambda}{h}\right)\mbox{(here, we used the fact that $\lambda^2=e\lambda+wh$)}\\
\end{array}
$$
Since $e \in (-\sqrt{D},-\sqrt{D}+21)$:
$$
\lambda-e=\frac{\sqrt{D}-e}{2} > \frac{2\sqrt{D}-21}{2} >\sqrt{D}-11.
$$
and
$$
0 < \frac{\lambda}{h}=\frac{e+\sqrt{D}}{4} < \frac{21}{4}
$$
which implies
 $$
 0 \leq m=\lfloor\frac{\lambda}{h}\rfloor \leq 5.
$$
Since $(m+1)-\frac{\lambda}{h} =(\lfloor\frac{\lambda}{h}\rfloor+1)-\frac{\lambda}{h} \geq \frac{1}{2}$ by the choice of $e$, we get
$$
\begin{array}{ccl}
 y & > &  \dfrac{1}{2}\times\dfrac{\sqrt{D}-11+2m}{2(m+2)} \\
   & > & \dfrac{1}{2}\times \left(\dfrac{\sqrt{D}-15}{2(m+2)}+1\right)\\
   & > & \dfrac{1}{2}\times \left( \dfrac{57-15}{2\times 7}+1\right) =2.
\end{array}
$$
This completes the proof of the proposition.
\end{proof}

\subsection{Proof of Theorem~\ref{thm:D1mod8:connect:PA1:PA2}}
\begin{proof}
By Lemma~\ref{lm:4cyl:dec:exist} and Proposition~\ref{prop:H6:noModA},
we know that every $\GL^+(2,\R)$-orbit in $\Omega E_D(6)$ contains a prototypical surface $X_D(w,h,t,e)$,
with $p=(w,h,t,e) \in \Pcal^A_D$.

If $D\not\in \{17,25,33,41,49,65,73,105\}$ and $D \not\in \mathrm{Exc}_2$, then Theorem~\ref{theo:H6:connect:PD} implies that $\Pcal^A_D$ has two components $\Pcal^{A_1}_D$ and $\Pcal^{A_2}_D$. It follows from Proposition~\ref{prop:D1mod8:connect:Dlarge} that there is a prototype in $\Pcal^{A_2}_D$ that is equivalent a prototype in $\Pcal^{A_1}_D$. Thus the theorem is proved for this case.

For $D\in \{17,25,33,49\}$, $\mathcal S^2_D$ is empty and $\Pcal_D$ has one component so there is nothing to prove.

For $D\in \{41,65,73,105\}$, and $D\in \mathrm{Exc}_2$ we can use the prototypes in $\Pcal^B_D$ and the switch moves to connect all the components of $\Pcal^A_D$. Details are given in Appendix~\ref{sec:D1mod8:except}.
\end{proof}


\section{Proof of the main theorems}\label{sec:prf:main:thms}
\subsection{Proof of Theorem~\ref{thm:H6:eig:comp}}
\begin{proof}
Let $(X,\omega)$ be a translation surface in $\Omega E_D(6)$ for a discriminant $D \geq 4$.
By Lemma~\ref{lm:4cyl:dec:exist} and Proposition~\ref{prop:normalize:A}, the $\GL^+(2,\R)$-orbit of $(X,\omega)$ contains a prototypical surface associated to some prototype $p$ in $\Pcal_D$.
Since $\Pcal_4=\Pcal_9=\varnothing$, the loci $\Omega E_4(6)$ and $\Omega E_9(6)$ are empty.

If $D=5$ then $\Pcal_D=\Pcal^B_D=\{(1,1,0,-1)\}$. Thus $\Omega E_5(6)=\GL^+(2,\R)\cdot X_5(1,1,0,-1)$.

Assume from now on that $D\neq 5$. Then by Proposition~\ref{prop:H6:noModA} we can take that $p\in \Pcal^A_D$.

\begin{itemize}

\item  Case $D\equiv 5 \mod 8$ follows from Theorem~\ref{theo:onto:map} and Theorem~\ref{theo:H6:connect:PD} (1).

\item Case $D \equiv 1 \mod 8$ and $D>9$ follows from Theorem~\ref{thm:D1mod8:connect:PA1:PA2}.

\item  Case $D$ even and $D >4$, by Theorems~\ref{thm:D:even:n:sq} and~\ref{thm:D:square:even} we get the desired conclusion for $\allowbreak D\not\in \{4,16,36,64,100,144\}$.
 For the remaining values of $D$ we have

 \begin{itemize}
 \item[.] {\bf $D=16$:} in  this case $\Pcal^A_{16}=\{(3,1,0,-2)\}$, thus $\Omega E_{16}(6)$ has one component.

 \item[.] {\bf $D=36$:} in this case $\Pcal^A_{36}$ has two components $\allowbreak \{(5,1,0,-4), (9,1,0,0)\}$ and $\allowbreak \{(8,1,0,-2)\}$.
 Consider the square-tiled   in Theorem~\ref{thm:D:square:even}, with $(l_A,l_B,l_C)=(2,1,1)$.
 This surface has a simple cylinder $C_1$ in the vertical direction of area $2$, and another simple cylinder $C_2$ in the direction of slope $\frac{3}{5}$ of area $3$.
 The prototype of the cylinder decomposition in the vertical direction is $(8,1,0,-2)$, and the prototype for the decomposition in the direction $\frac{3}{5}$ is $(9,1,0,0)$.
 Thus $\Omega E_{36}(6)$ has one component.

 \item[.] {\bf $D=64$:} $\Pcal^A_{64}$ has two components
 $$
 \left\{
 \begin{array}{ccl}
 \Pcal^{A^1}_{64} & = & \{(w,h,t,e) \in \Pcal^A_{64}, \ e \equiv 0 \mod 4\}, \\
 \Pcal^{A^2}_{64} & = &\{(w,h,t,e) \in \Pcal^A_{64}, \ e \equiv 2 \mod 4\}.
 \end{array}
 \right.
 $$
 Consider the square-tiled surface in Theorem~\ref{thm:D:square:even}, with $(l_A,l_B,l_C)=(2,2,1)$.
 This surface has a simple cylinder $C_1$ in the vertical  with $\Aa(C_1)=2$, and a simple cylinder $C_2$ in the direction of slope $\frac{2}{7}$ with $\Aa(C_2)=3$.
 The prototype of the cylinder decomposition in the vertical direction is $(\cdot,\cdot,\cdot,-4) \in \Pcal^{A^1}_{64}$, while the prototype of the decomposition in the direction of $C_2$ is $(\cdot,\cdot,\cdot,-2)\in \Pcal^{A_2}_{64}$.
 Thus $\Omega E_{64}(6)$ has only one component.

 \item[.] {\bf $D=100$:} $\Pcal^A_{100}$ has three components
 $$
 \left\{
 \begin{array}{ccl}
 \Pcal^{A^1}_{100} & = & \{(w,h,t,e) \in \Pcal^{A}_{100}, \, e\in \{-8,-4,0,4\}\}\\
 \Pcal^{A^2}_{100} & = & \{(16,1,0,-6),(12,2,1,-2),(14,1,0,2)\} \\
 \Pcal^{A^3}_{100} & = & \{(8,2,1,-6), (24,1,0,-2)\}.
 \end{array}
 \right.
 $$
 Let $(X,\omega)$ be the primitive square-tiled surface associated with the prototype $\allowbreak (24,1,0,-2) \in \Pcal^{A^3}_{100}$.
 By considering the cylinder decomposition in the direction of slope $\frac{1}{2}$, we see that $\GL^+(2,\R)\cdot(X,\omega)$ contains the square-tiled surface $(X',\omega')$ constructed in Theorem~\ref{thm:D:square:even} with $(l_A,l_B,l_C)=(4,1,2)$.
 We observe that $(X',\omega')$ has a simple cylinder in direction of slope $\frac{3}{8}$ of area $3$.
 The prototype of the corresponding cylinder  decomposition is $(\cdot,\cdot,\cdot,-4) \in \Pcal^{A^1}_{100}$.
 Thus the surfaces associated with prototypes in $\Pcal^{A^1}_{100}$ and $\Pcal^{A^3}_{100}$ belong to the same $\GL^+(2,\R)$-orbit.

 Consider now the square-tiled surface in Theorem~\ref{thm:D:square:even}, with $(l_A,l_B,l_C)=(2,3,1)$.
 This surface has a simple cylinder $C_1$ in the direction of slope $-2$ with $\Aa(C_1)=6$, and a simple cylinder $C_2$ in the direction of slope $\frac{2}{9}$ with $\Aa(C_2)=5$.
 The prototype of the cylinder decomposition in the direction of $C_1$ is $(14,1,0,2) \in \Pcal^{A^2}_{100}$, and the prototype of the decomposition in the direction of $C_2$ is $(25,1,0,0) \in \Pcal^{A^1}_{100}$. Thus $\Omega E_{100}(6)$ consists of a single $\GL^+(2,\R)$-orbit.

 \item[.] {\bf $D=144$:} we have $\Pcal^A_{144}$ has two components
 $$
 \left\{
 \begin{array}{ccl}
  \Pcal^{A^1}_{144} &  = & \{(w,h,t,e) \in \Pcal^A_{144}, \, e \equiv 0 \mod 4\}, \\
  \Pcal^{A^2}_{144} &  = & \{(w,h,t,e) \in \Pcal^A_{144}, \, e \equiv 2 \mod 4\}.
 \end{array}
 \right.
 $$
Consider  the square-tiled surface in Theorem~\ref{thm:D:square:even}, with $(l_A,l_B,l_C)=(2,4,1)$.
This surface has a simple cylinder $C_1$ in the vertical direction with $\Aa(C_1)=2$, and a simple cylinder $C_2$ in the direction of slope $\frac{2}{11}$ with $\Aa(C_2)=7$.
The prototype of the cylinder decomposition in the direction of $C_1$ is $(\cdot,\cdot,\cdot,-8) \in \Pcal^{A^1}_{144}$, and the prototype of the decomposition in the direction of $C_2$ is $(27,1,0,6) \in \Pcal^{A^2}_{144}$. Thus $\Omega E_{144}(6)$ consists of a single $\GL^+(2,\R)$-orbit.
\end{itemize}
The proof of the theorem is now complete.
\end{itemize}
\end{proof}

\subsection{Proof of Theorem~\ref{thm:H6:eig:comp:2}}
\begin{proof}
Theorem~\ref{thm:H6:eig:comp:2} is a direct consequence of Theorem~\ref{thm:H6:eig:comp} and Proposition~\ref{prop:square:tiled}.
\end{proof}

\appendix

\section{Proof of Theorem~\ref{theo:H6:connect:PD}}\label{sec:prf:connect:PD}

\subsection{Spaces of reduced  prototypes and almost-reduced prototypes}
The proof of Theorem~\ref{theo:H6:connect:PD} uses the {\em reduced  prototypes} and {\em almost  reduced prototypes} defined in Section~\ref{sec:red:prototypes}. It will be convenient to parametrize the set of reduced prototypes of discriminant $D$ by
$$
\Sc^1_{D}= \left\{ e\in \mathbb Z : e^2 \equiv D \mod 4 \textrm{ and } e^{2},\ (e+4)^{2} < D   \right\}.
$$
Similarly, when $D\equiv 1 \mod 8$, we will use the set
$$
\Sc^2_{D}  = \{  e\in \Z,\  e^{2} \equiv  D \mod  16,\ e^{2}, \textrm{ and } (e+8)^{2} < D \}.
$$
to parametrize the set of almost-reduced  prototypes.
For $h=1,2$, each element $e \in \SD$ gives rise to a prototype  $[e] := (w,h,0,e) \in \Pcal^A_D$, where $w:=(D-e^2)/4h$.

Recall that by Lemma~\ref{lm:reduced}, every component of $\Pcal_D^A$ contains an element of $\SD$.
As a consequence of Proposition~\ref{prop:BMpq:par:change}, we have
\begin{Lemma}\label{lm:BM:preserve:S2}
 Let $(w,2,0,e)$ be a prototype in $\Sc_D^2$. Let $q$ be a positive integer such that $\gcd(w/2,q)=1$, or $q=\infty$.
 If the Butterfly move $B_q$ is admissible then $B_q(w,2,0,e) \in \Sc_D^2$.
\end{Lemma}
\begin{proof}
Let $(w',h',t',e')=B_q(w,2,0,e)$. We first claim that $h'=2$.
Indeed, from Proposition~\ref{prop:BMpq:par:change}, we know that
$h'=\gcd(2q,w)$ if $q \in \N$, or $h'=\gcd(2,0)=2$ if $q=\infty$.
In the former case, since $\gcd(q,w/2)=1$ and $w$ is even we also have $h'=2$.

We now claim that both $w'$ and $t'$ is even. To see that, observe that the matrix
$\left(\begin{smallmatrix} a & b \\ c & d \end{smallmatrix}\right)$ in the proof of
Proposition~\ref{prop:BMpq:par:change} satisfies
$a\equiv b \equiv c \equiv d \equiv 0 \mod 2$. Since we have
$\left(
\begin{smallmatrix}
 w' & t' \\ 0 & h'
\end{smallmatrix}
\right)=\left(\begin{smallmatrix} 1 & -n \\ 0 & 1  \end{smallmatrix}\right)\cdot\left( \begin{smallmatrix} a & b \\ c & d \end{smallmatrix}\right) \cdot A
$ the claim follows.
Now, since $t' < \gcd(w',h')=2$, we must have $t'=0$, which means that $(w',h',t',e')\in \Sc_D^2$.
\end{proof}

\subsection{Connected components  of $\SD$}
We equip $\SD$ with the relation $e\sim e'$ if $[e']= B_q([e])$,
for some $q\in \N \cup \{\infty\}$ if $(e+4qh)^2<D$. Note that this condition implies that $e'=-e-4qh$, and  $\gcd(w,qh)=h$,
 when $q\in  \N\setminus \{0\}$, and  $e'=-e-4h$, when $q=\infty$.  An equivalence class  of the
equivalence  relation generated  by  this relation  is  called a  {\em  component} of $\SD$.

\begin{Theorem}[Components of $\SD$]
\label{theo:H6:connect:SD}
Let $D\geq 12$  be a  discriminant.  Let us assume that
$$ D   \not \in
\left\{
\begin{array}{l}
12,16,17,20,25,28,36,73,88,97,105,112,121,124,136,145,148, \\
169,172,184,193,196,201,217,220,241,244,265,268,292,304, \\
316, 364,385,436,484,556,604,676,796,844,1684
\end{array}
\right\}.
$$
Then the set ${\mathcal S}^1_{D}$ is non empty and has either
\begin{itemize}
\item three components, $\{e\in {\mathcal S}^1_{D},\ e\equiv 0 \textrm{ or } 4 \mod 8\}$, $\{e\in {\mathcal S}^1_{D},\ e\equiv 2 \mod 8\}$ and \\
$\{e\in {\mathcal S}_{D}^1,\ e\equiv -2 \mod 8\}$, if $D \equiv 4 \mod 8$,
\item two components,
\begin{itemize}
\item $\{e\in {\mathcal S}^1_{D},\ e\equiv 1 \textrm{ or } 3 \mod 8\}$ and $\{e\in {\mathcal S}^1_{D},\ e\equiv
-1 \textrm{ or } -3 \mod 8\}$ if $D \equiv 1 \mod 8$,
\item $\{e\in {\mathcal S}^1_{D},\ e\equiv 0 \textrm{ or } 4 \mod 8\}$ and $\{e\in {\mathcal S}^1_{D},\ e\equiv
+ 2 \textrm{ or } -2 \mod 8\}$ if $D \equiv 0 \mod 8$,
\end{itemize}
\item only one component, otherwise.
\end{itemize} \medskip

\noindent Let $D\geq 12$  be a  discriminant with $D\equiv 1 \mod 8$.  If
$$ D   \not \in
\left\{
17, 25, 33, 49,113, 145, 153, 177, 209, 217, 265, 273, 313, 321, 361, 385, 417, 481, 513
\right\}
$$
then the set ${\mathcal S}^2_{D}$ is non empty and connected.
\end{Theorem}
We follows the same strategy as the proof of~\cite[Th. 8.6]{Lanneau:Manh:H4}.
For the sake of completeness, we review the arguments here (that are slightly different),
and we do not wish to claim any originality in this part.

\subsection{Exceptional cases}

Our number-theoretic analysis of  the connectedness of $\SD$
only applies when $D$ is sufficiently large ({\it e.g.} $D \geq (83h)^{2}$). On
one  hand  it is  feasible  to compute  the  number  of components  of
$\SD$ when  $D$ is reasonably small. This  reveals the exceptional    cases   of   Theorem~\ref{theo:H6:connect:SD}.
 On  the other hand, using computer assistance, one can easily
prove the following

\begin{Lemma}
\label{lm:exceptionnal:cases}
Theorem~\ref{theo:H6:connect:SD} is true for all $D \leq (83h)^{2}$.
\end{Lemma}


\subsection{Small values of $q$} Surprisingly it is possible to show
that Theorem~\ref{theo:H6:connect:SD}  holds for most values of  $D$ only by
using   butterfly   moves  $B_q$   with   small   $q$,  namely   $q\in
\{1,2,3,5,7\}$. If $q$ is a prime  number, we will use the following two
operations
$$ \left\{
\begin{array}{lll}
F_{q}(e) &=& e + 4h(q-1)=B_\infty B_q([e]), \\ F_{-q}(e) &=& e - 4h(q-1)=B_qB_\infty([e]).
\end{array}
\right.
$$


These two maps are useful to us, since we have
\begin{Proposition}
Let $e\in \SD$, and assume that $q$ is an odd prime.
\begin{enumerate}
\item If $F_{q}(e)  \in \SD$ and $D  \not \equiv e^{2} \mod
  q$ then $e \sim F_{q}(e)$.
\item If $F_{-q}(e)  \in \SD$ and $D  \not \equiv (e+4h)^{2}
  \mod q$ then $e \sim F_{-q}(e)$.
\end{enumerate}
\end{Proposition}
\begin{proof}
It  suffices  to  remark  that  $[F_q(e)]$  (resp.  $[F_{-q}(e)]$)  is
obtained   from   $[e]$   by   the   sequence   of   butterfly   moves
($B_q,B_\infty$)   (resp.   ($B_\infty,B_q$)),   and  the   respective
conditions ensure the admissibility of the corresponding sequence
(and $\gcd(w,qh)=h$ since $w$ is even if $h=2$).
\end{proof}

\noindent The  next   proposition  guarantees  that,  under   some  rather  mild
assumptions, one has $e \sim F_3(e)=e+8h$.

\begin{Proposition}
\label{prop:equiv:step8}
Let $e\in  \SD$  and let us  assume that $e-24h$  and $e+32h$
also belong to $\SD$. Then one of the following two holds:
\begin{enumerate}
\item $e \sim e+8h$, or
\item  $(D,e)$ is  congruent  to $(4h^2,-10h)$  or  $(4h^2,-2h)$
  modulo $105 = 3\cdot5\cdot7$.
\end{enumerate}
\end{Proposition}

\begin{proof}
We say that a sequence of integers $(q_1,q_2,\dots,q_n)$ is a strategy
for $(D,e)$ if for any $i=1,\dots,n-1$ the following holds:
$$
\left\{
\begin{array}{l}
e_1=e,\\
q_i \textrm{ is admissible for } (D,e_i), \\
e_{i+1}=F_{q_i}(e_i)   \in   e   +  \{-24h,-16h,-8h,0,8h,16h,24h,32h\}, \\
 e_{n} = e+8h.
\end{array}\right.
$$
For instance,  if $(D,e) \equiv (0,3) \mod 105$  then $(5,-3)$ is a
strategy. Indeed  letting $e = 3$ we  see that $3 \sim  F_{5}(3) = 19$
since $5$  is admissible for $(D,3)$.  And $19 \sim F_{-3}(19)  = 11 =
3+8$   since  $-3$  is   admissible  for   $(D,19)$.  Hence   $3  \sim
3+8$.  \medskip

Thus in order to prove the proposition we only need to give a strategy
for every pair $(D,e) \mod 105$  with the two exceptions stated in the
theorem. In fact each of the  $105^2-2$ cases can be handled by one of
the following $12$ strategies.

\begin{enumerate}
\item There  are $7350$  pairs $(D,e)$ for  which $q=3$  is admissible
  ({\it i.e.}  $D\not \equiv  e^2 \mod 3$).  Since $F_3(e) =  e+8h$ the
  sequence $(3)$ is a strategy for all of these cases.

\item Among the $105^2-2-7350=3673$  remaining pairs, there are $1960$
  pairs $(D,e)$ for which the sequence $(5,-3)$ is a common strategy.

\item We can continue searching  strategies for all remaining
  pairs $(D,e)$ but two: $(4h^2,-10h)$ and $(4h^2,-2h)$.
  We found  the following strategies:
$$
\begin{array}{l}
(7,-5),\  (-3,5),  \ (-5,7),  \\  (5,3,-5),  \  (-5,3,5), \  (5,5,-7),
 \ (-7,5,5), \ (-3,7,-3), \\ (-5,3,7,-3), \ (-3,7,3,-5).
\end{array}
$$
\end{enumerate}
Note  that the  condition  that  $e-24h$ and  $e+32h$  belong to  $\SD$
guarantees the admissibility of the strategies. This completes the proof
of the proposition.
\end{proof}

\begin{Remark}
Since for  $(D,e) \equiv  (4h^2,-2h) \mod 105$  one has $D \equiv e^2 \equiv (e+4h)^2
\mod  105$, even  though one can enlarge the set  of primes to be
used in the strategies, there is no hope to get a similar conclusion to Proposition~\ref{prop:equiv:step8}
without the second case.
\end{Remark}

\begin{Remark}
\label{rk:crietrion:ends}
A simple criterion to be not  close to the ends of $\SD$ is
the following.
\begin{center}
If $f\in \SD$ then for any $e>f,  \qquad (e+36h < \sqrt{D})
\implies (e+32h\in \SD)$.
\end{center}
Indeed, $e+32h\in \SD$ if and only if $(e+32h)^{2} <
D$ and $(e+32h+4h)^{2}=(e+36h)^{2} < D$. Thus the claim is obvious if $e+32h\geq0$.
Now, if $e < -32h$, then since $e>f$ the inequalities
$$
0> e + 32h > f + 32h > f \qquad \textrm{and} \qquad -(f+4h)> 4h> e + 36h
> f + 36h > f + 4h
$$
implies
$$ (e + 32h)^{2} < f^{2} <  D \qquad \textrm{and} \qquad (e + 36h)^{2} <
(f + 4h)^{2} < D.
$$
\end{Remark}
Let us define $\mathcal  T^h_{D} = \{  e \in \SD, \ e-24h \textrm{  and }
e+32h\in \SD \}$. Simple calculations show
\begin{Lemma}\label{lm:TD:dual}
 Assume that $D>(36h)^2$. Then if $e\in \mathcal{T}^h_D$ and $e \geq -2h$, then $-e-4h \in \mathcal{T}^h_D$.
\end{Lemma}

The next proposition asserts that if $D$ is large then assumption of Proposition~\ref{prop:equiv:step8}
actually holds.

\begin{Proposition}
\label{prop:ends}
Assume that $D\geq (55h)^2$ if $h=1$ and $D\geq (63h)^2$ if $h=2$.
Then every element of $\SD$ is equivalent to an element of $\mathcal T^h_{D}$.
\end{Proposition}
\begin{proof}[Proof of Proposition~\ref{prop:ends}]
Let $f\in \SD$. Since $f \sim -f - 4h$ we can assume $f \leq
-2h$.  If  $f>-6h$ then  the proposition is  clearly true,  therefore we
only have to consider the case $f\leq -6h$. Observe that if $ f\leq -6h$
then  $(f+32h)^{2} \leq (f-20h)^{2}$  and $(f+36h)^{2}  \leq (f-24h)^{2}$.
In particular $f  - 24h \in  \SD$ implies $f  + 32h \in \SD$. On the other hand,
since $f<0$, $f  - 24h \not \in  \SD$ if and only if $D \leq (f-24h)^{2}$.
Thus assume that
\begin{equation}
\label{eq:f}
f ^{2} < D \leq (f-24h)^{2}.
\end{equation}
We will  show that there always exists $e\in \SD$, $e\sim f$ with $e>f$
and  $e+36h <  \sqrt{D}$, which implies that $e+32h\in \SD$  by
Remark~\ref{rk:crietrion:ends}. If $e-24h \not \in \SD$ then
by definition, $e$ satisfies  the inequalities~(\ref{eq:f}) and thus we
can repeat the argument by replacing $f$ by $e$.

\medskip

If $h=1$ then $ D\geq 55^2$ and we have $f\leq 24-55=-31$.
If $h=2$ then $ D > (62h)^2$ and we have $f< 24h-62h=-38h=-76$.

If $h=1$, assume that  there exists prime $q  \leq  13$ such  that $\gcd(w,q)  =
1$. Then $f \sim F_{q}(f) > f$ and
\begin{equation*}
F_q(f)+36= f+4(q-1)+36 \leq -31 +48+36=53 < 55 \leq \sqrt{D}.
\end{equation*}
Hence $e = F_{q}(f)$ is convenient if $h=1$. Thus we may assume that
$w$  is divisible  by all primes $p\leq 13$. Thus
$D  \geq 4 \cdot  w  \geq 4 \cdot 2 \cdot 3 \cdot 5 \cdot 7 \cdot 11 \cdot 13 > 10^{5}$.

The same applies if $h=2$: assume that  there exists some {\em odd}
prime $q  \leq  17$ such  that $\gcd(w,2q)  = 2$. Then $f \sim F_{q}(f) > f$ and
\begin{equation*}
F_q(f)+36h= f+4h(q-1)+36h \leq -38h +64h+36h=62h < 63h \leq \sqrt{D}.
\end{equation*}
Hence $e = F_{q}(f)$ is convenient if $h=2$. Thus we may assume that
$w$  is divisible  by all {\em odd} primes $p\leq 17$. Thus (recall that $w$
is even if $h=2$):
$D  \geq 4 \cdot  w  \geq 4 \cdot 2\cdot 3 \cdot 5 \cdot 7 \cdot 11 \cdot 13 \cdot 17 > 10^{6}$.
By~\cite[Theorem~9.1]{Mc4} there is  an integer  $q$
relatively prime to $w$ such that
$$ 1 < q < \frac{3\log(w)}{\log(2)} \leq 5\log(D).
$$ Now $f\sim F_{q}(f)$ where
$$ f < F_{q}(f) = f + 4h(q-1) < 20h \cdot \log(D).
$$ Since for $D \geq 10^{5}$ if $h=1$ and $D \geq 10^{6}$ if $h=2$, we have
$$ F_q(f)+36h < 20h \cdot \log(D) + 36h < \sqrt{D}.
$$ This completes the proof of Proposition~\ref{prop:ends}.
\end{proof}
\subsection{Case $D \equiv 4h^2 \mod 105$}
Proposition~\ref{prop:equiv:step8} implies that if $D\not\equiv 4h^2 \mod 105$
then $e \in \mathcal{T}^h_D \Rightarrow e \sim e+8h$.
We now handle the case $D \equiv 4h^2\mod 105$.

We define
$$ \mathcal U^h_{D} =  \{ e \in \mathcal T^h_{D}, \ e  \not \equiv -2h \mod
105\},
$$

\begin{Lemma}
\label{lm:eq:UD}
Assuming $D \equiv 4h^2\mod 105$.
For  $D  >  (83h)^{2}  =  6889\cdot h^2$,  all elements  of  $\SD$  are
equivalent to an element of $\mathcal U^h_{D}$.
\end{Lemma}
\begin{proof}
Let $e\in \SD$. Since $D>(83h)^2$, Proposition~\ref{prop:ends}
implies  that one  can assume  $e\in  \mathcal T^h_{D}$.  Let us  assume
$e\not \in  \mathcal U^h_{D}$, {\it i.e.}   $e \equiv -2h  \mod 105$.
By Lemma~\ref{lm:TD:dual}, one can  assume $e\leq  -2h$.
To prove the lemma, we need the following

\begin{Lemma}
\label{eq:claim}
Let $w=\frac{D-e^2}{4h}$. For $D > (83h)^{2}$, $D\equiv 4h^2 \mod 105$, and $e \in \mathcal{T}^h_D, \, e \leq -2h$,
there exists $q\in \N$ such that
$$
q\not \equiv 1 \mod 105, \quad \gcd(w,q) = 1, \textrm{ and } \quad  4qh + 31h < \sqrt{D}.
$$
\end{Lemma}

Let us first complete the proof of Lemma~\ref{lm:eq:UD}. According to Lemma~\ref{eq:claim}, we can pick some  $q\in \N$  such that
$$
\gcd(w,q)=1 \text{ and } F_{q}(e) + 36h = e+4h(q-1) + 36h = e + 4qh + 32h \leq 4qh + 30h < \sqrt{D}
$$
Thanks to Remark~\ref{rk:crietrion:ends},  we know that $F_{q}(e) +36h< \sqrt{D}$ implies $F_{q}(e) + 32h \in \SD$.
Since $e\in \mathcal{T}^h_D$, it follows that $F_{q}(e) \in \mathcal T^h_{D}$.
Since $F_{q}(e) - e \equiv 4h(q-1) \not\equiv  0  \mod  105$, we  have  $F_{q}(e)  \not  \equiv -2  \mod  105$,
i.e.  $F_{q}(e)  \in  \mathcal  U_{D}$.
We conclude  by  noting  that if $h=2$ then $\gcd(w,q) = 1$ implies $\gcd(w,qh)=2$,
which implies $e\sim F_{q}(e)$. Of course if $h=1$ the same conclusion
applies. Lemma~\ref{lm:eq:UD} is now proved.
\end{proof}

To complete the proof of our statement, it remains to show

\begin{proof}[Proof of Lemma~\ref{eq:claim}]
One has to show that there exists $q \in \N$ such that
\begin{equation}
\label{eq:lemma}
\left\{ \begin{array}{l}  \gcd(w,q) = 1,  \\ q\not \equiv 1  \mod 105,
  \\ 4qh + 31h < \sqrt{D}.
\end{array} \right.
\end{equation}
Since  $D > (83h)^{2}$  the last  two conditions  of~(\ref{eq:lemma}) are
automatic for $q =  2,\ 3,\ 5,\ 7,\ 11,\ 13$. Thus one can assume
$w$  is divisible  by all of these primes, otherwise the lemma is proved.
For both values of $h$, we have $wh \geq 2\cdot 3\cdot 5\cdot 7 \cdot 11\cdot 13 = 30030$,
thus $\sqrt{D}=\sqrt{e^{2}+4\cdot w\cdot  h} > 346$.
\medskip

Again, the last two conditions  are fulfilled for all primes less than
$73$ (odd primes if $h=2$); thus  the claim  is proven  unless $w$ is  divisible by  all of
these $21$ primes, in which case we have $w > 10^{28}$.
\medskip

To find   a   good    $q$   satisfying   the   first   condition
of~(\ref{eq:lemma}),  we will  use the  Jacobsthal's  function $J(n)$,
that  is  defined  to  be  largest gap  between  consecutive  integers
relatively prime  to $n$.   A convenient estimate
for  $J(n)$ is provided  by Kanold:  If none  of the
first   $k$   primes   divide   $n$,   then   one   has   $J(n)   \leq
n^{\log(2)/\log(p_{k+1})}$,   where    $p_{k+1}$   is   the   $(k+1)th$
prime.
\medskip

We will also use  the following inequality that can be found
in~\cite{Mc4}   (Theorem~$9.4$):  \\  For   any  $a,n,w\geq   1$  with
$\gcd(a,n) =  1$ there is a  positive integer $q \leq  n J(w//n)$ such
that
$$
q \equiv a \mod n \textrm{ and } \gcd(q,w) = 1,
$$
where $w//n$  is obtained  by removing  from $w$  all  primes that
divide $n$. \bigskip

Applying  the above inequality with $a=13$ and  $n=210$, one can
find a positive integer $q$ satisfying
$$
q \leq 210J(w//210), \quad \gcd(w, q) = 1, \textrm{ and } q \equiv 13 \mod 210.
$$
In particular $q\not \equiv   1  \mod  105$,   and  thus  the  first   two  conditions
of~(\ref{eq:lemma})   are  satisfied.   Let  us   see  for   the  last
condition. \medskip

\noindent Since the  first prime $p_{k+1}$ that divide  $w//210$ is $13$, Kanold's estimates gives
$$ J(w//210) \leq (w//210)^{\log(2)/\log(p_{k+1})} \leq (w//210)^{1/3}
\leq w^{1/3}.
$$
Hence
$$
4\cdot qh + 31h \leq 4\cdot 210h \cdot w^{1/3} + 31h.
$$
But since $w > 10^{28}$ and $D\geq 4w$, we have:
$$
4\cdot 210h \cdot w^{1/3} + 31h < w^{1/2}h \leq \sqrt{D}.
$$
The lemma is proved.
\end{proof}


\begin{proof}[Proof of Theorem~\ref{theo:H6:connect:SD} when $h=1$]
We will assume that $D\geq 83^2$ (by Lemma~\ref{lm:exceptionnal:cases}).
Thanks  to Proposition~\ref{prop:ends},  every component  of $\mathcal
S^1_{D}$ meets $\mathcal T^1_{D}$. Since $D=e^2+4w$ the possible values of
$D$ modulo $8$ are
$$
D \equiv 0,1,4,5 \mod 8.
$$
We will examine each case separately.

We define
$$
\begin{array}{lll}
\mathcal T^{1,i}_{D} = \{ e \in  \mathcal T^1_{D},\ e \equiv 2i \mod 8\} & i=0,1,2,3 & \textrm{ if } D \textrm{ is even,} \\
\mathcal T^{1,i}_{D} = \{ e \in \mathcal T^1_{D},\ e \equiv  1+2i \mod 8\} &i=0,1,2,3 & \textrm{ if }  D \textrm{ is odd.}  \\
\end{array}
$$
We first assume $D\not \equiv 4 \mod 105$.
By Proposition~\ref{prop:equiv:step8} we  have  $e  \sim  e +  8$
whenever $e$ is in $\mathcal T^1_{D}$.  Therefore, all
elements of  ${\mathcal T}^{1,i}_{D}$ are equivalent for $i=0,1,2,3$.
Thus Proposition~\ref{prop:ends} implies
${\mathcal S}^1_{D}$ has at most four components.
Now, for each values of $D \mod 8$, we
connect the sets ${\mathcal T}^{1,i}_{D}$ together.
\begin{enumerate}
\item If $D\equiv 0  \mod 8$ then $0\in {\mathcal T}^{1,0}_{D}$ is connected to $B_1(0) = 0-4\times 1 =
-4  \in  {\mathcal  T}^{1,2}_{D}$,   and  $-10\in  {\mathcal  T}^{1,3}_{D}$  is
connected to $B_2(-10) = 10-4\times  2 =  2 \in {\mathcal T}^{1,1}_{D}$ (observe that $w=(D-(-10)^2)/4$ is odd
so that $B_2(-10) \in \mathcal S^1_D$).

\item if $D\equiv 1  \mod 8$ then $1\in {\mathcal T}^{1,0}_{D}$ is  connected to
$B_1(1)= -1-4 \times  1 =  -5\in {\mathcal T}^{1,1}_{D}$, and $5\in {\mathcal T}^{1,2}_{D}$ is  connected to
$B_1(5)\sim -5-4\times 1 = -9\in {\mathcal T}^{1,3}_{D}$.

\item If $D\equiv 4  \mod 8$ then $0\in  {\mathcal
T}^{1,0}_{D}$ is connected to $B_1(0)= -4 \in {\mathcal T}^{1,2}_{D}$.

\item If $D\equiv 5  \mod 8$ then $1\in  {\mathcal
T}^{1,0}_{D}$ is connected to $B_1(1)= -5 \in {\mathcal T}^{1,1}_{D}$, and
$5\in  {\mathcal T}^{1,2}_{D}$ is  connected to  $B_1(5)= -9  \in {\mathcal
T}^{1,3}_{D}$. Finally,  $1\in {\mathcal T}^{1,0}_{D}$ is connected to
$B_2(1,1)= -9 \in {\mathcal T}^{1,3}_{D}$ since $\frac{D-1^2}{4}$ is odd.

\end{enumerate}

We now assume $D\equiv 4 \mod 105$.
Recall  that in  this  case we  have  defined $\mathcal{U}^1_D:=\{e  \in
\mathcal{T}^1_D,  \; e  \not\equiv -2  \mod 105\}$.  We consider the partition of
$\mathcal{U}^1_D$ by $\mathcal U^{1,i}_{D} = \mathcal U^1_{D} \cap \mathcal T^{1,i}_{D}$.
It is easy to check that all  elements  of $\mathcal  U^{1,i}_{D}$  are  equivalent in  $\mathcal S^1_{D}$.
Indeed we can apply Proposition~\ref{prop:equiv:step8}.  Since $D  \equiv 4
\mod  105$ and $e  \not \equiv  -2 \mod  105$, if  we can  not conclude
directly that  $e \sim e+8$  then this means  that $e \equiv  -10 \mod
105$. But in this case, since
$$ e^{2} \equiv 0 \not \equiv D \equiv 1 \mod 5
$$  one can  apply the  move $F_{q}$  with $q=5$.  This gives  $e \sim
F_{5}(e) = e + 16$. This proves the lemma. \medskip

Now by Lemma~\ref{lm:eq:UD} we  only need
to   connect    elements   in   $\mathcal{U}^{1,i}_D$,   with   different
$i$.  Actually,  we  can  use  the  same butterfly moves as above (when $D
\not\equiv 4  \mod 105$)  since they do  not involve any  element $e\in
\mathcal{T}^{1,i}_D$ such that $e \equiv  -2 \mod 105$. This completes the
proof of Theorem~\ref{theo:H6:connect:SD} when $h=1$.
\end{proof}

\begin{proof}[Proof of Theorem~\ref{theo:H6:connect:SD} when $h=2$]
Again we will assume that $D\geq (83\cdot 2)^2$ (by Lemma~\ref{lm:exceptionnal:cases}).
Thanks  to Proposition~\ref{prop:ends},  every component  of $\mathcal
S^2_{D}$ meets $\mathcal T^2_{D}$. Recall that if $e\in \mathcal S^2_D$ then $e^2\equiv D \mod 16$.
Set
$$
\begin{array}{lll}
\mathcal T^{2,i}_{D} = \{ e \in  \mathcal T^2_{D},\ e \equiv 1+2i \mod 16\} & i=0,3,4,7 & \textrm{ if } D \equiv 1 \mod 16, \\
\mathcal T^{2,i}_{D} = \{ e \in \mathcal T^2_{D},\ e \equiv  1+2i \mod 16\} &i=1,2,5,6 & \textrm{ if }  D \equiv 9 \mod 16.  \\
\end{array}
$$
We first assume $D\not \equiv 4\times 2^2=16 \mod 105$.
By Proposition~\ref{prop:equiv:step8} we  have  $e  \sim  e +  16$
whenever $e$ is in $\mathcal T^2_{D}$.  Therefore all
elements of  ${\mathcal T}^{2,i}_{D}$ are equivalent.
Thus Proposition~\ref{prop:ends} implies
${\mathcal S}^2_{D}$ has at most four components.
Now we connect the elements of
${\mathcal T}^{2,i}_{D}$.

If $D\equiv 1 \mod 16$ then
\begin{enumerate}
\item $1\in {\mathcal T}^{2,0}_{D}$ is  connected to $B_1(1)= -1-4 \times 2  = -9 \in {\mathcal T}^{2,3}_{D}$.
\item $9\in {\mathcal T}^{2,4}_{D}$ is  connected to $B_1(9)= -9-4 \times 2  = -17 \in {\mathcal T}^{2,7}_{D}$.
\item Set $w_1=\frac{D- 1^2}{16}$ and $w_9=\frac{D-9^2}{16}$. Since we have $w_1-w_9=5$, one of $w_1$ and $w_9$ is odd.
If follows that we can apply the Butterfly move $B_2$ to either $1\in \mathcal{T}^{2,0}_D$ or  $9\in \mathcal{T}^{2,4}_D$.
In the first case, we get $B_2(1)=-1-16=-17 \in {\mathcal T}^{2,7}_{D}$ and
in the second case $B_2(9)=-9-16=-25\in {\mathcal T}^{2,3}_{D}$. Hence all the sets $\mathcal{T}^{2,i}$ are connected.
\end{enumerate}

We now turn to the case  $D\equiv 9 \mod 16$.
\begin{enumerate}
\item $3\in {\mathcal T}^{2,1}_{D}$ is  connected to $B_1(3)= -3-4 \times 2  = -11 \in {\mathcal T}^{2,2}_{D}$.
\item $11\in {\mathcal T}^{2,5}_{D}$ is  connected to $B_1(11)= -11-4 \times 2  = -19 \in {\mathcal T}^{2,6}_{D}$.
\item Set $w_3=\frac{D-3^2}{16}$ and $w_{11}=\frac{D-11^2}{16}$. Since $w_3-w_{11}=7$, one of $w_3$ and $w_{11}$ is odd.
Thus one can apply $B_2$ to either $3\in \mathcal{T}^{2,1}_D$ or $11\in \mathcal{T}^{2,5}_D$.
In the first case we draw $B_2(3)=-3-16=-19 \in {\mathcal T}^{2,6}_{D}$ and
in the second case $B_2(11)=-11-16=-27\in {\mathcal T}^{2,2}_{D}$.
Hence  all the $\mathcal{T}^{2,i}$ are connected.
\end{enumerate}

If $D\equiv 4 \mod 105$, we apply the same idea as in the case $h=1$.
We consider the partition of  $\mathcal{U}^2_D$ by $\mathcal U^{2,i}_{D} = \mathcal U^2_{D} \cap \mathcal T^{2,i}_{D}$.
All elements of $\mathcal  U^{2,i}_{D}$ and we can connect elements   in   $\mathcal{U}^{2,i}_D$,    with   different
$i$. We  can  use  the  same butterfly moves as above, i.e. when $D \not\equiv 4  \mod 105$,
since they do  not involve any  element $e\in \mathcal{T}^{2,i}_D$ such that $e \equiv  -2\cdot 2 = -4 \mod 105$.
This completes the proof of Theorem~\ref{theo:H6:connect:SD} when $h=2$.
\end{proof}

\subsection{Components of $\Pcal^A_D$: proof of Theorem~\ref{theo:H6:connect:PD}}
By Theorem~\ref{theo:H6:connect:SD}, we only need to discuss three cases
$$
 \left\{ \begin{array}{l} D \geq 12 \qquad \textrm{and} \qquad D \equiv 4 \mod 8.  \\
D \geq 17 \qquad \textrm{and} \qquad D \equiv 1 \mod 8 \\
 D \textrm{ belongs to the sets of exceptional discriminants of Theorem~\ref{theo:H6:connect:SD}}.
\end{array}
\right.
$$
We  first examine  the  generic  cases.

\subsubsection{Proof of Theorem~\ref{theo:H6:connect:PD} when $D\equiv 4 \mod 8$}

Since any element of $\Pcal^A_D$ is equivalent to an element of $\mathcal S^1_D$
if is sufficient to connect  the two components
$\allowbreak \{ e \in \mathcal{S}^1_D, \ e \equiv 2 \mod 8\}$ and
$\allowbreak \{e \in \mathcal{S}^1_D, \ e \equiv -2 \mod 8 \}$
of $\mathcal S^1_D$  by using non-reduced elements of
$\mathcal P^A_{D}$.

\begin{itemize}
\item If $D=12+16k$ (with $k\geq 2$), then $q=2$ is admissible for $e=-2$ and
$$
[-2] \overset{B_{2}}{\longrightarrow} (2k-3,2,0,-6) \overset{B_{\infty}}{\longrightarrow} (2k+1,2,0,-5)
\overset{B_{1}}{\longrightarrow} [-6]
$$
connects the two components since $-6\equiv +2 \mod 8$.

\item If $D =  4 + 32k$. One can assume $k\geq 4$ since $D\not \in \{36,68,100\}$.
Hence $q=2$ is admissible for $e=2$ and
$$
[2] \overset{B_{2}}{\longrightarrow} (4k-12,2,1,-10) \overset{B_{2}}{\longrightarrow} (4k-4,2,1,-6)
\overset{B_{1}}{\longrightarrow} [-2]
$$
connects the two components.

\item If $D =  20 + 32k$. One can assume $k\geq 3$ since $D\not \in \{52,84\}$.
Hence $q=2$ is admissible for $e=2$ and
$$
[2] \overset{B_{2}}{\longrightarrow} (4k-10,2,1,-10) \overset{B_{2}}{\longrightarrow} (2k-1,4,0,-6)
\overset{B_{1}}{\longrightarrow} [-10]
$$
connects the two components since $-10\equiv -2 \mod 8$.
\end{itemize}

\subsubsection{Proof of Theorem~\ref{theo:H6:connect:PD} when $D\equiv 1 \mod 8$}
Recall that for $h=1,2$, we have
$$
\Pcal_D^{A_h}:=\{p=(w,h,t,e) \in \Pcal_D^A, \; \mbox{the equivalence class of $p$ contains an element of } \Sc_D^h\}
$$
and $\Pcal_D^A=\Pcal_D^{A_1}\sqcup \Pcal_D^{A_2}$ (see Lemma~\ref{lm:reduced} and Lemma~\ref{lm:D1mod8:S1:no:connect:S2}).

By Theorem~\ref{theo:H6:connect:SD} $\Sc^1_D$ contains two components $\{e \in \Sc^1_D, \; e \equiv 1,3 \mod 8\}$ and
$\allowbreak \{e \in \Sc^1_D, \ e \equiv -1,-3 \mod 8\}$.
We show that those two components  can be connected through $\Pcal_D^{A_1}$.
\begin{itemize}
\item If $D=1+16k$ (with $k\geq  3$), then $[-5]\in {\mathcal S}^1_{D}$. Thus
$$
[-5] \overset{B_{2}}{\longrightarrow} (2k-1,2,0,-3) \overset{B_{\infty}}{\longrightarrow} (2k-3,2,0,-5)
\overset{B_{1}}{\longrightarrow} [-3]
$$
connects the two components since $-5\equiv 3 \mod 8$.

\item If $D =  9 + 16k$ (with $k\geq 3$ since $D\neq 41$), then
$[-7]\in {\mathcal S}^1_{D}$. Thus
$$
[-7] \overset{B_{2}}{\longrightarrow} (2k+1,2,0,-1) \overset{B_{\infty}}{\longrightarrow} (2k-5,2,0,-7)
\overset{B_{1}}{\longrightarrow} [-1]
$$
connects the two components since $-7\equiv 1 \mod 8$.
\end{itemize}
Since $\Sc^2_D$ contains a single component by Theorem~\ref{theo:H6:connect:SD}, this proves the theorem for non exceptional values of $D$.
\subsubsection{Proof of Theorem~\ref{theo:H6:connect:PD} for $D$ in the sets of exceptional discriminants of Theorem~\ref{theo:H6:connect:SD}}
\label{sec:except:0}

The strategy is to connect ``extra'' components of $\mathcal S^1_D$ and $\mathcal S^2_D$  by using moves through $\mathcal  \Pcal^A_D$.

We first prove the statement on the components of $\Pcal_D^A$, and $\Pcal^{A_1}_D$ if $D \equiv 1 \mod 8$, for
$$ D \in
\left\{
\begin{array}{l}
12,16,17,20,25,28,36,73,88,97,105,112,121,124,136,145,148, \\
169,172,184,193,196,201,217,220,241,244,265,268,292,304, \\
316, 364,385,436,484,556,604,676,796,844,1684
\end{array}
\right\}.
$$
For $D\in\{12,16,17,25\}$, one can check by hand that $\Sc^1_D$ has only one component.
For $D \in \{20,28,36\}$, $\mathcal S^1_{D}$ has two components ${\{-2\} \textrm{ and } \{-4,0\}}$. So there is nothing to prove.

The first non trivial discriminant to discuss is $D=73$. We directly check that  $\mathcal S^1_{73}$ has three components, namely
$\{-5,1\}, \{-7,3\},\{-3,-1\}$. We can connect them through $\Pcal^A_{73}$ by
$$
\begin{array}{l}
{ [-5] \overset{B_3}{\longrightarrow} (2,3,0,-7) \overset{B_{\infty}}{\longrightarrow} (4, 3, 0, -5)\overset{B_{1}}{\longrightarrow} [-7]}\\
{ [-7] \overset{B_2}{\longrightarrow} (9, 2, 0, -1) \overset{B_{\infty}}{\longrightarrow} (3, 2, 0, -7)\overset{B_{1}}{\longrightarrow} [-1]}\\
\end{array}
$$
(recall that for $e\in \mathcal S^1_D$ we define $[e]=(w,1,0,e)\in \mathcal \Pcal^A_D$,
where $w=(D^{2}-e^{2})/4$).

The next discriminant to consider is $D=88$. This time $\mathcal S^1_{88}$ has three components, namely
$\{0, -4\}, \{-8, 4\}, \{2, -6, -2\}$. We can connect $\{0, -4\}$ and $\{-8, 4\}$ through $\Pcal^A_{88}$ by
$$
{[-8] \overset{B_{4}}{\longrightarrow} (3, 2, 0, -8)\overset{B_{1}}{\longrightarrow} [0]}
$$
This proves Theorem~\ref{theo:H6:connect:PD} for $D=88$.
Using computer assistance, we can repeat the above discussion for all the remaining discriminants
$$ D   \in
\left\{
\begin{array}{l}
97,105,112,121,124,136,145,148, 169,172,184,193,196,201,217,220,241,\\
244,265,268,292,304, 316, 364,385,436,484,556,604,676,796,844,1684\\
\end{array}
\right\}
$$
to show   that
\begin{itemize}
\item $\Pcal^A_D$ has one component when $D \equiv 5$ $\mod 8$,
\item $\Pcal^A_D$ has two components when $D \equiv 0$ or $4$ $\mod 8$,
\item $\Pcal_D^{A_1}$ has one component   when $D \equiv 1$ $\mod 8$.
\end{itemize}

We now turn to the statement on $\Pcal^{A_2}_D$, for
$$
D \in
\left\{
17, 25, 33, 49,113, 145, 153, 177, 209, 217, 265, 273, 313, 321, 361, 385, 417, 481, 513
\right\}
$$
We check directly that $\mathcal S^2_{D}$ is empty for $D \in \{17, 25, 33, 49\}$. The other components are given in the table below.
$$
\begin{array}{|c|c||c|c|}
\hline
D & \textrm{components of }\mathcal S^2_D &D & \textrm{components of }\mathcal S^2_D \\
\hline
113& {\scriptstyle\{-7, -1\}, \{1, -9\}} & 313& {\scriptstyle\{3, -11\}, \{-13, -3, 5, -5\} }\\
145& {\scriptstyle\{-7, -1\}, \{1, -9\}} &321& {\scriptstyle\{1, -7, -9, 9, -1, -17\}, \{-15, 7\} }\\
153& {\scriptstyle\{3, -11\}, \{-5, -3\}} &361& {\scriptstyle\{3, -3, -11, -5\}, \{-13, 5\} }\\
177& {\scriptstyle\{-7, -1\}, \{1, -9\}} &385&{\scriptstyle \{1, -15, 7, -7, -9, -1\}, \{9, -17\}} \\
209& {\scriptstyle\{-7, -1\}, \{1, -9\}} &417& {\scriptstyle\{1, -7, -9, 9, -1, -17\}, \{-15, 7\} }\\
217& {\scriptstyle\{3, -11\}, \{-13, -3, 5, -5\}} &481& {\scriptstyle\{1, -15, 7, -7, -9, -1\}, \{9, -17\}} \\
265& {\scriptstyle\{3, -3, -11, -5\}, \{-13, 5\}} &513& {\scriptstyle\{1, -7, -9, 9, -1, -17\}, \{-15, 7\}} \\
273& {\scriptstyle\{1, -9\}, \{-15, -7, -1, 7\}} && \\
\hline
\end{array}
$$
For instance, for $D=217$ one can connect the two components of $\mathcal S^2_{217}$ through $\Pcal^A_{217}$ by
$$
{[-13] \overset{B_{3}}{\longrightarrow} (4, 6, 0, -11)\overset{B_{\infty}}{\longrightarrow} (2, 6, 0, -13)\overset{B_{1}}{\longrightarrow}[-11]. }
$$
(here, for $e\in \mathcal S^2_D$ we define $[e]=(w,2,0,e)\in \mathcal \Pcal^A_D$, where $w=(D^{2}-e^{2})/8$). This shows that $\Pcal^{A_2}_{217}$ has one component proving Theorem~\ref{theo:H6:connect:PD} for this case.
Again, we easily show by using computer assistance that $\Pcal^{A_2}_D$ is non empty and has one component for
$$
D \in \{217, 273, 321, 361, 385, 417, 513\}
$$
For the discriminants in $\mathrm{Exc}_2:=\{113,145, 153, 177, 209,265,313,481\}$, $\Pcal^{A_2}_D$ actually has two components.

\section{Exceptional values of $D$}

In this section, we discuss the particular cases of Theorem~\ref{thm:D1mod8:connect:PA1:PA2} and Theorem~\ref{thm:D:even:n:sq}. To this purpose, we will need several tools that we detail in the coming section.

\subsection{Tools for exceptional values of $D$}
\label{sec:except:tools}

\begin{Lemma}\label{lm:connect:PA:PB:D:nonsq}
Fix a discriminant $D$ which is not a square.
Let $(X,\omega)=X_D(w,h,0,e)$ be the prototypical surface associated with a prototype $p=(w,h,0,e)$ in either $\Sc^1_D$ or $\Sc^2_D$.
Then $(X,\omega)$ admits a cylinder decomposition in Model B in the direction $\theta$ with slope $\frac{h}{\lambda}$ (see Figure~\ref{fig:reducedA:to:B}).
Let $(w',h',t',e')\in\Pcal^B_D$ be the prototype of the corresponding cylinder decomposition. Then we have
\begin{itemize}
  \item If $(w,h,0,e) \in \Sc^1_D$, that is $h=1$ and $w=\frac{D-e^2}{4}$, then
  \begin{equation}\label{eq:S1:to:B}
  \frac{\lambda'}{w'}=\frac{\lambda+n}{\lambda+n+1},
  \end{equation}
  where $n=\lfloor \frac{w}{\lambda} \rfloor =\lfloor \frac{\sqrt{D}-e}{2} \rfloor$.
  \item If $(w,h,0,e) \in \Sc^2_D$, that is $h=2$ and $w=\frac{D-e^2}{8}$ is even, then
  \begin{equation}\label{eq:S2:to:B}
  \frac{\lambda'}{w'}=\frac{2\lambda+2(2n+\epsilon)}{2\lambda+2(2n+1+\epsilon)},
  \end{equation}
  where $n=\lfloor \frac{w}{\lambda}\rfloor=\lfloor \frac{\sqrt{D}-e}{4}\rfloor$, and
  $$
  \epsilon=\left\{
  \begin{array}{cl}
  1  & \text{ if } \frac{w}{\lambda} -\lfloor \frac{w}{\lambda} \rfloor > \frac{1}{2},\\
  0  & \text{ if } \frac{w}{\lambda} -\lfloor \frac{w}{\lambda} \rfloor < \frac{1}{2}.
  \end{array}
  \right.
  $$
\end{itemize}
\end{Lemma}

\begin{Remark}
  If $D$ is a square, $(X,\omega)$ may have a two-cylinder decomposition in the direction $\theta$.
\end{Remark}

\begin{figure}[htb]
\begin{minipage}[t]{0.4\linewidth}
\centering
\includegraphics[width=3.5cm]{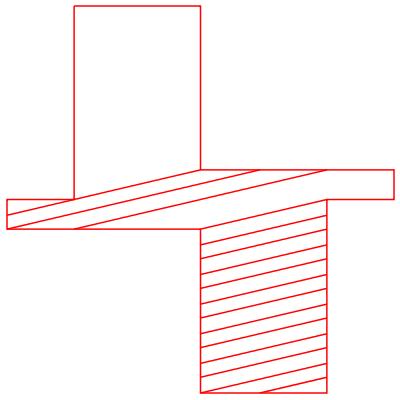}
\end{minipage}
\begin{minipage}[t]{0.4\linewidth}
\centering
\includegraphics[width=3.5cm]{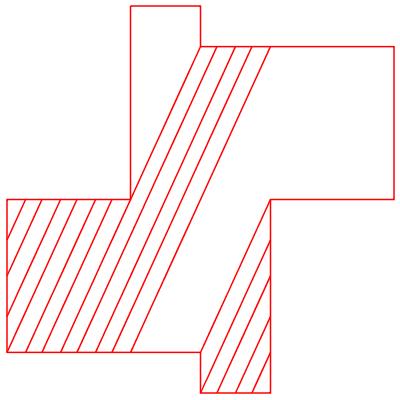}
\end{minipage}

  \caption{Cylinders in direction with slope $\frac{h}{\lambda}$ for $p\in \Sc^1_D$ (left) and $p\in \Sc^2_D$ (right).}
  \label{fig:reducedA:to:B}
\end{figure}

\begin{proof}
Since $D$ is not a square, $(X,\omega)$ cannot admit a two-cylinder decomposition. Hence the cylinder decomposition in the direction $\theta$ is either in Model A or Model B.
Consider the  saddle connection $\delta_0$ in the direction $\theta$ which passes through the unique regular fixed point of the Prym involution of $X$.
There are exactly two saddle connections in the direction $\theta$ with length half of $\delta_0$, namely $\delta_1,\delta_2$.
If the corresponding cylinder decomposition is of Model A, then we must have four such saddle connections.
Therefore, we can conclude that this decomposition is of Model B.

\medskip

Let $p'=(w',h',e',t')$ be the prototype in $\Pcal^B_D$ of the cylinder decomposition in the direction $\theta$.
Consider the saddle connection $\delta$ whose union with $\delta_1$ is a boundary component of a semi-simple cylinder in the direction $\theta$.
Comparing with the prototypical surface in Proposition~\ref{prop:normalize:A}, we get
$$
\frac{w'-\lambda'}{\lambda'}=\frac{|\delta_1|}{|\delta_1|+|\delta|}=\frac{(\delta_1)_y}{(\delta_1)_y+(\delta)_y}.
$$
where $(\alpha)_y$ stands for the $y$-component of the holonomy vector of the saddle connection $\alpha$.
The formulas~\eqref{eq:S1:to:B} and \eqref{eq:S2:to:B} then follow from a careful inspection of the number of times $\delta$ crosses each horizontal cylinder.
\end{proof}

We introduce now some more switch moves to connect a prototype in $\Pcal^B_D$ with other prototypes.
In what follows, $(X,\omega)$ is the prototypical surface corresponding to a prototype $p=(w,h,t,e)$ in $\Pcal^B_D$.
\begin{Lemma}[$S_5$ move]\label{lm:S5:move}\hfill
\begin{enumerate}
\item If $t=0$, then $(X,\omega)$ admits a cylinder decomposition in Model B in the vertical direction with prototype $(w',h',0,e')$, where
$$
e'=3e+4h-2w.
$$
\item If $t\neq 0$ and $\lambda+e+2h-w-t >0$, then $(X,\omega)$ admits a cylinder decomposition in Model A in the direction  of slope $\frac{h+\lambda}{t}$. Let $(w',h',t',e')$ be the prototype of this cylinder decomposition. Then we have
    $$
    e'=3e+4h-2w-2t.
    $$
\end{enumerate}
In both cases, we will call the prototype $(w',h',t',e')$ the transformation of $(w,h,t,e)$ by the $S_5$ move.
\end{Lemma}

\begin{Lemma}[$S_6$ move]\label{lm:S6:move}
The surface $(X,\omega)$ always admits a  4-cylinder decomposition in the direction of slope $\frac{\lambda+h}{\lambda+t}$.
Let $(w',h',e',t)$ be the prototype of this cylinder decomposition.
\begin{enumerate}
\item If $w+t-2h-e>0$, then $(w',h',t',e') \in \Pcal^A_D$, and
$$
e'=3e+4h-2w.
$$
\item If $w+t-2h-e<0$, then $(w',h',t',e') \in \Pcal^A_D$, and
$$
e'=e+2t.
$$
\item If $w+t-2h-e=0$, then $(w',h',t',e') \in \Pcal^B_D$, and
$$
e'=e+2t.
$$
The prototype $(w',h',t',e')$ will be called the transformation of $(w,h,t,e)$ by the $S_6$ move.
\end{enumerate}
\end{Lemma}

\begin{Lemma}[$S_7$ move]\label{lm:S7:move}
Assume that $\lambda > w-h-t$. Then $(X,\omega)$ admits a 4-cylinder decomposition in the direction of slope $\frac{\lambda+h}{w-t}$. Let $(w',h',t',e')$ be the prototype of this cylinder decomposition.
\begin{enumerate}
  \item If $t < e+h$ then $(w',h',t',e') \in \Pcal^A_D$, and
  $$
  e'=e+2h-2w+2t.
  $$
  \item If $t > e+h$ then $(w',h',t',e') \in \Pcal^A_D$, and
  $$
  e'=3e+4h-2w.
  $$
  \item If $t = e+h$ then $(w',h',t',e') \in \Pcal^B_D$, and
  $$
  e'=3e+4h-2w.
  $$
\end{enumerate}
The prototype $(w',h',t',e')$ will be called the transformation of $(w,h,t,e)$ by the $S_7$ move.
\end{Lemma}

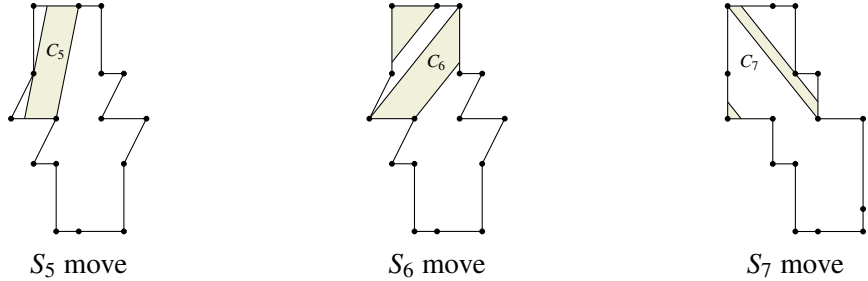
\begin{figure}[htb]
\begin{minipage}[t]{0.3\linewidth}
\centering
\begin{tikzpicture}[scale=0.3]
 \fill[blue!30!yellow!20] (0.6,5) -- (2,5) -- (3,10) -- (1.6,10) -- cycle;
 \draw (0,5) -- (2,5) -- (1,3) -- (2,3) -- (2,0) -- (5,0) -- (5,3) -- (6,5) -- (4,5) -- (5,7) -- (4,7) -- (4,10) -- (1,10) -- (1,7) -- cycle;
 \draw (0.6,5) -- (1.6,10) (2,5) -- (3,10);
 \foreach \x in {(0,5),(1,10), (1,7), (1,3), (2,5), (2,3), (2,0), (3,10), (3,0), (4,10), (4,7), (4,5), (5,7), (5,3), (5,0), (6,5)} \filldraw[fill=black] \x circle (3pt);
 \draw (2,8) node {\tiny $C_5$};
 \draw (3,-1.5) node {$S_5$ move};
\end{tikzpicture}
\end{minipage}
\begin{minipage}[t]{0.3\linewidth}
\centering
\begin{tikzpicture}[scale=0.3]
 \fill[blue!30!yellow!20] (0,5) -- (2,5) -- (4,7.5) -- (4,10) -- cycle;
 \fill[blue!30!yellow!20] (1,7.5) -- (3,10) -- (1,10) -- cycle;
 \draw (0,5) -- (2,5) -- (1,3) -- (2,3) -- (2,0) -- (5,0) -- (5,3) -- (6,5) -- (4,5) -- (5,7) -- (4,7) -- (4,10) -- (1,10) -- (1,7) -- cycle;
 \draw (0,5) -- (4,10) (2,5) -- (4,7.5) (1,7.5) -- (3,10);
 \foreach \x in {(0,5),(1,10), (1,7), (1,3), (2,5), (2,3), (2,0), (3,10), (3,0), (4,10), (4,7), (4,5), (5,7), (5,3), (5,0), (6,5)} \filldraw[fill=black] \x circle (3pt);
 \draw (3,7.5) node {\tiny $C_6$};
 \draw (3,-1.5) node {$S_6$ move};
\end{tikzpicture}
\end{minipage}
\begin{minipage}[t]{0.3\linewidth}
\centering
\begin{tikzpicture}[scale=0.3]
 \fill[blue!30!yellow!20] (4,5) -- (4,5.75) -- (0.6,10) -- (0,10) -- cycle;
 \fill[blue!30!yellow!20] (0,5) -- (0.6,5) -- (0,5.74) -- cycle;
 \draw (0,5) -- (2,5) -- (2,3) -- (3,3) -- (3,0) -- (6,0)  -- (6,5) -- (4,5) -- (4,7) -- (3,7) -- (3,10) -- (0,10) -- cycle;
 \draw (4,5) -- (0,10) (0.6,10) -- (4,5.75) (0,5.75) -- (0.6,5);
 \foreach \x in {(0,5),(0,10), (0,7), (2,3), (2,5), (2,10), (3,0), (3,10), (3,3), (3,7),  (4,7), (4,5), (4,0), (6,5), (6,), (6,0)} \filldraw[fill=black] \x circle (3pt);
 \draw (1,7.5) node {\tiny $C_7$};
 \draw (3,-1.5) node {$S_7$ move};
\end{tikzpicture}
\end{minipage}
\caption{The switch moves $S_5,S_6,S_7$}
\label{fig:switch:567}
\end{figure}

\subsection{Proof of Theorem~\ref{thm:D:even:n:sq} for $D\in \{52,68,84\}$}
\label{sec:except:1}
\begin{proof}\hfill

\noindent {\bf Case $D=52$.}
The three components of ${\mathcal S}^1_{52}$ are $\{-6,2\}$, $\{-4,0\}$ and $\{-2\}$.
We have $[-2]=(12, 1, 0, -2) \in \mathcal S^1_{52}$.
By Lemma~\ref{lm:connect:PA:PB:D:nonsq}, $(X,\omega)$ admits a cylinder decomposition in Model B, with prototype $(w',h',t',e')$ where $\frac{\lambda'}{w'}=\frac{\lambda+n}{\lambda+n+1}$ and $n=\lfloor \frac{w}{\lambda} \rfloor=4$. Direct computation shows that $(w',h',t',e')=(3, 4, 0, -2) \in \Pcal^B_{52}$.  This connects prototype $[-2]$ to $(3, 4, 0, -2)$.
Now the moves $S_2$ is admissible and we have $S_2(3, 4, 0, -2)=(9, 1, 0, -4)=[-4]$.  This connects $[-2]$ and $[-4]$.

We have $[2]=(12, 1, 0, 2) \in \mathcal S^1_{52}$.
By Lemma~\ref{lm:connect:PA:PB:D:nonsq}, $(X,\omega)$ admits a cylinder decomposition in Model B, with prototype $(w',h',t',e')$ where $\frac{\lambda'}{w'}=\frac{\lambda+n}{\lambda+n+1}$ and $n=\lfloor \frac{w}{\lambda} \rfloor=2$. Direct computation shows that $(w',h',t',e')=(3, 4, 0, -2) \in \Pcal^B_{52}$.  This connects prototype $[2]$ to $(3, 4, 0, -2)$ as desired. \medskip

\noindent {\bf Case $D=68$.}
The three components of ${\mathcal S}^1_{68}$ are $\{-6,2\}$, $\{-4,0,4\}$ and $\{-2\}$. We have $[-2]=(16,1,0,-2) \in \mathcal S^1_{68}$.
By Lemma~\ref{lm:connect:PA:PB:D:nonsq}, $(X,\omega)$ admits a cylinder decomposition in Model B, with prototype $(w',h',t',e')$ where $\frac{\lambda'}{w'}=\frac{\lambda+n}{\lambda+n+1}$ and $n=\lfloor \frac{w}{\lambda} \rfloor=5$. Direct computation shows that $(w',h',t',e')=(8,1,0,6) \in \Pcal^B_{68}$.  This connects prototype $[-2]$ to $(8,1,0,6)$.
Now the moves $S_2$ and $S_4$ are admissible and we have
$S_2(8,1,0,6)=(13, 1, 0, 4)=[4]$ and $S_4(8,1,0,6)=(16, 1, 0, 2)=[2]$. This connects the three components together as desired. \medskip

\noindent {\bf Case $D=84$.}
The three components of ${\mathcal S}^1_{84}$ are $\{-6,2\}$, $\{-8,-4,0,4\}$ and $\{-2\}$.
We have $[-2]=(20, 1, 0, -2) \in \mathcal S^1_{84}$.
By Lemma~\ref{lm:connect:PA:PB:D:nonsq}, $(X,\omega)$ admits a cylinder decomposition in Model B, with prototype $(w',h',t',e')$ where $\frac{\lambda'}{w'}=\frac{\lambda+n}{\lambda+n+1}$ and $n=\lfloor \frac{w}{\lambda} \rfloor=5$. Direct computation shows that $(w',h',t',e')=(4, 5, 0, -2) \in \Pcal^B_{84}$.  This connects prototype $[-2]$ to $(4, 5, 0, -2)$.
Now the moves $S_2$ is admissible and we have $S_2(4, 5, 0, -2)=(17, 1, 0, -4)=[-4]$. This connects $[-2]$ and $[4]$.

We have $[2]=(20, 1, 0, 2) \in \mathcal S^1_{84}$.
By Lemma~\ref{lm:connect:PA:PB:D:nonsq}, $(X,\omega)$ admits a cylinder decomposition in Model B, with prototype $(w',h',t',e')$ where $\frac{\lambda'}{w'}=\frac{\lambda+n}{\lambda+n+1}$ and $n=\lfloor \frac{w}{\lambda} \rfloor=3$. Direct computation shows that $(w',h',t',e')=(4, 5, 0, -2) \in \Pcal^B_{84}$.  This connects prototype $[2]$ to $(4, 5, 0, -2)$ as desired.
\end{proof}
\subsection{Theorem~\ref{thm:D1mod8:connect:PA1:PA2} for $D\in \{41,65,73,105\}$ and $D\in\mathrm{Exc}_2$}
\label{sec:D1mod8:except}
\begin{proof}\hfill

\noindent {\bf Case $D=41$.} We first observe that  $\Sc^1_{41}$ has two components $\{(4,1,0,-5),(10,1,0,1)\}$ and $\allowbreak \{(8,1,0,-3),(10,1,0,-1)\}$, while $\Sc^2_{41}$ has only one component $\{(2,2,0,-5),(4,2,0,-3)\}$.

Set $\allowbreak p_1=(10,1,0,1), p_2=(10,1,0,-1), p_3=(2,2,0,-5)$.
By Proposition~\ref{prop:H6:noModA}, any $\GL^+(2,\R)$-orbit in $\Omega E_{41}(6)$ contains a prototypical surface associated to a prototype in $\Pcal^A_{41}$.
By Lemma~\ref{lm:reduced}, any prototype in $\Pcal^A_{41}$ is equivalent to one of $\{p_1,p_2,p_3\}$.
Using Lemma~\ref{lm:connect:PA:PB:D:nonsq}, we see that for all $i\in \{1,2,3\}$, $p_i$ is equivalent to either $q_1=(2,4,0,-3)$ or $q_2=(2,4,1,-3)$.
Note that both $q_1,q_2$ are elements of $\Pcal^B_{41}$.
But we have the following relations
$$\left\{
\begin{array}{ccccc}
(2,4,0,-3) \in \Pcal^B_{41} & \overset{S_5}{\longrightarrow} & (8,1,0,3)\in \Pcal^B_{41} & \overset{S_3}{\longrightarrow} & (10,1,0,1) \\
(2,4,1,-3) \in \Pcal^B_{41} & \overset{S_5}{\longrightarrow} & (10,1,0,1). & &
\end{array}
\right.
$$
Thus the locus $\Omega E_{41}(6)$ contains a single $\GL^+(2,\R)$-orbit.

\medskip

\noindent {\bf Case $D=65$.}  One can easily check that $\Pcal^{A_2}_{65}$ contains exactly two prototypes $\{(2,2,0,-7),(8,2,0,-1)\}$. From Lemma~\ref{lm:connect:PA:PB:D:nonsq}, we see that both prototypes in $\Pcal^{A_2}_{65}$ is equivalent to one prototype in the following family
$$
\{(4,4,0,-1),(4,4,1,-1),(4,4,2,-1),(4,4,3,-1)\}.
$$
Set $q_i=(4,4,i,-1), \; i=0,\dots,3$. Note that $q_i\in \Pcal^B_{65}$ for all $i$.
We have the following relations
$$
\left\{
\begin{array}{l}
q_0 \overset{S_2}{\longrightarrow} (\cdot,\cdot,\cdot,-3)\in \Pcal^{A_1}_{65}\\
q_1 \overset{S_5}{\longrightarrow} (14,1,0,3) \in \Pcal^{A_1}_{65}\\
q_2 \overset{S_6}{\longrightarrow} (14,1,0,3) \in \Pcal^{A_1}_{65}\\
q_3 \overset{S_6}{\longrightarrow} (10,1,0,5) \in \Pcal^B_{65} \overset{S_2}{\longrightarrow} (\cdot,\cdot,\cdot,-3) \in \Pcal^{A_1}_{65}.
\end{array}
\right.
$$
By Theorem~\ref{theo:H6:connect:PD} we know that $\Pcal^{A_1}_{65}$ contains a single component.
Thus the proposition is proved for this case.

\medskip
\noindent {\bf Case $D=73$.} In this case $\Pcal^{A_2}_{73}$ contains exactly two prototypes $\{(6,2,0,-5),(8,2,0,-3)\}$.
By Lemma~\ref{lm:connect:PA:PB:D:nonsq}, we see that both elements of $\Pcal^{A_2}_{73}$ are equivalent to a prototype in the family
$$
\{(2,6,0,-5),(2,6,1,-5)\}
$$
We have the following relations
$$
\left\{
\begin{array}{l}
(2,6,0,-5)\in \Pcal^B_{73} \overset{S_2}{\longrightarrow} (\cdot,\cdot,\cdot,-7) \in \Pcal^{A_1}_{73}\\
(2,6,1,-5) \in \Pcal^B_{73} \overset{S_7}{\longrightarrow} (12,1,0,5) \in \Pcal^B_{73} \overset{S_2}{\longrightarrow} (\cdot,\cdot,\cdot,-7) \in \Pcal^{A_1}_{73}.
\end{array}
\right.
$$
Since $\Pcal^{A_1}_{73}$ has only one component by Theorem~\ref{theo:H6:connect:PD}, the proposition is proved for this case.

\medskip

\noindent {\bf Case $D=105$.} In this case $\Pcal^{A_2}_{105}$ contains exactly two prototypes $\{(10,2,0,-5),(12,2,0,-3)\}$.
By Lemma~\ref{lm:connect:PA:PB:D:nonsq}, both elements of $\Pcal^{A_2}_{73}$ are equivalent to a prototype in the family
$$
\{(4,6,0,-3),(4,6,1,-3)\}
$$
We have the following relations
$$
\left\{
\begin{array}{l}
(4,6,0,-3)\in \Pcal^B_{105} \overset{S_4}{\longrightarrow} (\cdot,\cdot,\cdot,-7) \in \Pcal^{A_1}_{105}\\
(4,6,1,-3) \in \Pcal^B_{105} \overset{S_6}{\longrightarrow}  (\cdot,\cdot,\cdot,-1) \in \Pcal^{A_1}_{105}.
\end{array}
\right.
$$
Again, we conclude by Theorem~\ref{theo:H6:connect:PD}.

\bigskip

We finish the proof for
$$
D\in \mathrm{Exc}_2=\{113,145, 153, 177, 209,265,313,481\}.
$$
The strategy is the same: $\Pcal^{A_2}_{D}$ contains exactly two components. From Lemma~\ref{lm:connect:PA:PB:D:nonsq}, one sees that both components is equivalent to some prototypes in $\Pcal^B$. We then use the Switch moves $S_i$ for $i=1,\dots,7$ to connect these prototypes to $\Pcal^{A_1}_{D}$. This last step is easily done by a direct computation.
Since by Theorem~\ref{theo:H6:connect:PD}, $\Pcal^{A_1}_{D}$ contains a single component, this finishes the proof of the theorem.
\end{proof}



\end{document}